\documentclass[12pt]{article}
\usepackage{amssymb,url,color}

\newtheorem{theorem}{Theorem}[section]
\newtheorem{prop}[theorem]{Proposition}
\newtheorem{corollary}[theorem]{Corollary}

\newtheorem{preqn}{Question}
\newenvironment{question}{\preqn\rm}{\endpreqn}
\newtheorem{unnumber}{}

\newtheorem{prepf}[unnumber]{Proof}
\newenvironment{proof}{\prepf\rm}{\endprepf}
\newcommand{\qed}{\hfill$\Box$}

\newcommand{\Aut}{\mathop{\mathrm{Aut}}\nolimits}

\newcommand{\Com}{\mathop{\mathrm{Com}}\nolimits}
\newcommand{\DCom}{\mathop{\mathrm{DCom}}\nolimits}
\newcommand{\Pow}{\mathop{\mathrm{Pow}}\nolimits}
\newcommand{\DPow}{\mathop{\mathrm{DPow}}\nolimits}
\newcommand{\EPow}{\mathop{\mathrm{EPow}}\nolimits}
\newcommand{\Gen}{\mathop{\mathrm{Gen}}\nolimits}
\newcommand{\NGen}{\mathop{\mathrm{NGen}}\nolimits}
\newcommand{\Nilp}{\mathop{\mathrm{Nilp}}\nolimits}
\newcommand{\Sol}{\mathop{\mathrm{Sol}}\nolimits}
\newcommand{\Cay}{\mathop{\mathrm{Cay}}\nolimits}
\newcommand{\DEP}{\mathop{\mathrm{DEP}}\nolimits}
\newcommand{\Forb}{\mathop{\mathrm{Forb}}\nolimits}

\begin{document}
\title{Graphs defined on groups}
\author{Peter J. Cameron\\University of St Andrews\\
\texttt{pjc20@st-andrews.ac.uk}}
\date{}
\maketitle

\begin{abstract}
These notes concern aspects of various graphs whose vertex set is a group $G$
and whose edges reflect group structure in some way (so that, in particular,
they are invariant under the action of the automorphism group of $G$). The
particular graphs I will chiefly discuss are the power graph, enhanced power
graph, deep commuting graph, commuting graph, and non-generating graph.

My main concern is not with properties of these graphs individually, but 
rather with comparisons between them. The graphs mentioned, together
with the null and complete graphs, form a hierarchy (as long as $G$ is
non-abelian), in the sense that the edge set of any one is contained in that
of the next; interesting questions involve when two graphs in the hierarchy
are equal, or what properties the difference between them has. I also 
consider various properties such as universality and forbidden subgraphs,
comparing how these properties play out in the different graphs. (There are
so many papers even on the power graph that a complete survey is scarcely
possible.)

I have also included some results on intersection graphs of subgroups of
various types, which are often in a ``dual'' relation to one of the other
graphs considered. Another actor is the Gruenberg--Kegel graph, or prime graph,
of a group: this very small graph has a surprising influence over various
graphs defined on the group. I say little about Cayley graphs, since (except
in special cases) these are not invariant under the automorphism group of $G$.

Other graphs which have been proposed, such as the nilpotence, solvability, 
and Engel graphs, will be touched on rather more briefly. My emphasis is on
finite groups but there is a short section on results for infinite groups.

Proofs, or proof sketches, of known results have been included where possible.
Also, many open questions are stated, in the hope of stimulating further
investigation.

The graphs I chiefly discuss all have the property that they contain
\emph{twins}, pairs of vertices with the same neighbours (save possibly one
another). Being equal or twins is an equivalence relation, and the
automorphism group of the graph has a normal subgroup inducing the symmetric
group on each equivalence class. For some purposes, we can merge twin vertices
and get a smaller graph. Continuing until no further twins occur, the result
is unique independent of the reduction, and is the trivial $1$-vertex graph if
and only if the original graph is a \emph{cograph}.  So I devote a section to
cographs and twin reduction, and another to the consequences for automorphism
groups. In addition, I discuss the question of deciding, for each type of
graph, for which groups is it a cograph. Even for the groups
$\mathrm{PSL}(2,q)$, this leads to difficult number-theoretic questions.

There are briefer discussions of general $\Aut(G)$-invariant graphs,
and structures other than groups (such as semigroups and rings).
\end{abstract}

\tableofcontents

\section{Introduction}

There are a number of graphs whose vertex set is a group $G$ and whose edges
reflect the structure of $G$ in some way, so that the automorphism group of
$G$ acts as automorphisms of the graph. These include the commuting graph
(first studied in 1955), the generating graph (from 1996), the power graph
(from 2000), and the enhanced power graph (from 2007), all of which have a
considerable and growing literature. A relative newcomer, not published
yet, is the deep commuting graph.

This paper does not aim to be a survey of all these areas, which would be
far too ambitious a task. Rather, I am interested in comparisons among the
different graphs. In particular, there is a hierarchy containing the
null graph, power graph, enhanced power graph, deep commuting graph, 
commuting graph, non-generating graph (if the group is non-abelian), and
complete graph: the edge set of each is contained in that of the next.

These graphs have some similarities: for example, the enhanced power graphs,
commuting graphs, deep commuting graphs, and generating graphs of finite
groups all form universal families (that is, every finite graph is embeddable
in one of these graphs for some group $G$). However, the proofs of this
require rather different techniques for the different graphs.

Another question, about which relatively little is currently known, concerns
the differences between graphs in the hierarchy. Even rather basic questions
such as connectedness are unstudied for most of these, although Saul Freedman
and coauthors have results on the difference between the non-generating
graph and the commuting graph (and, at top and bottom, the difference between
the complete graph and the non-generating graph is the generating graph, while
the difference between the power graph and the null graph is the power graph,
both of which have an extensive literature).

For some of these graphs, either the complementary graph was defined first,
or the graph and its complement were studied independently. For example,
the generating graph of a finite group preceded the non-generating graph;
and Neumann's theorem, that if the non-commuting graph of an infinite group
contains no infinite clique then it has finite clique number, clearly has
an equivalent formulation in terms of the coclique number of the commuting
graph. Also, several of the questions I consider have easy translations for
the complement graph; for example, the universality results just mentioned
immediately show that the complementary graphs are universal too. Twin
reduction and the related concept of cograph do not distinguish between a
graph and its complement. I have chosen to focus on those graphs which form
a hierarchy; if the complements were preferred, the hierarchy would reverse.

A curious feature is the appearence of the Gruenberg--Kegel graph, which
determines (or almost determines) various features of the commuting graph
and the power graph. The vertex set of this graph is not the group, but the
much smaller set of prime divisors of the group order. For example, if $G$
has trivial centre, its reduced commuting graph (with the identity removed)
is connected if and only if its Gruenberg--Kegel graph is. Conversely, for all
graph types in the hierarchy except possibly the non-generating graph, the 
corresponding graph on $G$ determines the Gruenberg--Kegel graph of $G$.

Authors who have studied these have used a variety of notations for them. I
have tried to use a consistent and helpful notation, for example, 
$\Pow(G)$ and $\Com(G)$ for the power graph and commuting graph, respectively,
of~$G$.

The final, brief sections concern more general graphs defined on groups and
invariant under group automorphisms; the graphs of the hierarchy on
infinite groups; and extensions to other algebraic structures such as
semigroups and rings.

Since much is not known, I have tried to emphasise open problems throughout.

Computations reported here were performed using \textsf{GAP}~\cite{GAP}, with
the packages GRAPE~\cite{grape} for handling graphs, HAP~\cite{hap}
for computing Schur and Bogomolov multipliers, and LOOPS~\cite{loops}
for Moufang loops. Generators for specific
groups were taken from the On-Line Atlas of Finite Groups~\cite{atlas}.

I'm grateful to several people, especially Alireza Abdollahi, Saul Freedman,
Michael Giudici, Michael Kinyon, Bojan Kuzma and Natalia Maslova, for helpful
comments on a previous version.

\subsection{Notation}

I will denote a typical graph by $\Gamma$, with vertex set $V(\Gamma)$ and
edge set $E(\Gamma)$. If $A$ is a subset of $V(\Gamma)$, then the
\emph{induced subgraph} of $\Gamma$ on $A$ is the graph with vertex set $A$
whose edges are those of $\Gamma$ contained in $A$. A \emph{complete graph}
is one in which all pairs of vertices are joined; a \emph{null graph} is one
with no edges. A \emph{clique} (resp.\ \emph{coclique}) is a set of vertices
on which the induced subgraph is complete (resp.\ null).

For a graph $\Gamma$,
\begin{itemize}\itemsep0pt
\item the \emph{clique number} $\omega(\Gamma)$ is the size of the largest
clique;
\item the \emph{clique cover number} $\theta(\Gamma)$ is the minimum number
of cliques whose union is the vertex set of $\Gamma$;
\item the \emph{independence number} or \emph{coclique number} $\alpha(\Gamma)$
is the size of the largest coclique;
\item the \emph{chromatic number} $\chi(\Gamma)$ is the smallest number of
independent sets whose union is the vertex set of $\Gamma$ (so-called because
it is the smallest number of colours needed to colour the vertices so that
adjacent vertices get different colours).
\end{itemize}
Note that the independence number and clique cover number of $\Gamma$ are
equal to the clique number and chromatic number of the complement of $\Gamma$
(the graph with the same vertex set, whose edges are the non-edges of $\Gamma$).

It is clear that $\omega(\Gamma)\leqslant\chi(\Gamma)$ and
$\alpha(\Gamma)\leqslant\theta(\Gamma)$. The graph $\Gamma$ is called
\emph{perfect} if every induced subgraph has clique number equal to chromatic
number. The \emph{Weak Perfect Graph Theorem} of Lov\'asz~\cite{lovasz}
asserts that, if $\Gamma$ is perfect, then so is its complement; the
\emph{Strong Perfect Graph Theorem} of Chudnovsky \emph{et~al.}~\cite{crst}
asserts that $\Gamma$ is perfect if and only if it does not contain a cycle
of odd length greater than $3$ or the complement of one as an induced
subgraph.

The \emph{comparability graph} of a partial order $P=(A,{\leqslant})$ is the
graph with vertex set $A$, in which $a$ and $b$ are joined if either
$a\leqslant b$ or $b\leqslant a$. \emph{Dilworth's Theorem}~\cite{dilworth}
 asserts that the comparability graph of a partial order and its complement
are both perfect. (The first part is easy, the second less so but follows from
the first using the Weak Perfect Graph Theorem.)

Groups will almost always be finite here; my notation for finite groups is
standard. (The reason for calling a graph $\Gamma$ is to avoid conflict of
notation, since $G$ will be a typical group.)

\subsection{Cayley graphs}
\label{s:cayley}

One topic I will not consider, except in this section, concerns Cayley graphs.
A \emph{Cayley graph} for the group $G$ is a graph on the vertex set $G$
which is invariant under right translation by elements of $G$. (Some authors
use left translation; the two concepts are equivalent, and the inversion map
on $G$ converts one into the other.) Equivalently, if $S$ is an inverse-closed
subset of $G\setminus\{1\}$, then the Cayley graph $\Cay(G,S)$ is the graph
with vertex set $G$ in which $g$ and $h$ are adjacent whenever $gh^{-1}\in S$.

One reason for not considering these is that they have a huge literature, far
more than I can survey here. I have heard the view expressed that algebraic
graph theory is the study of Cayley graphs of finite groups (in fact, it is
broader than this, but Cayley graphs are an important topic); while
it is certainly arguable that geometric group theory is the study of Cayley
graphs of finitely generated infinite groups.

The other is that Cayley graphs are not in general preserved by the 
automorphism group of $G$. I will say a few words about this.

Suppose that the set $S$ is a \emph{normal subset} of $G$, that is,
closed under conjugation. Then $\Cay(G,S)$ is invariant under both left and
right translation. Such a graph is sometimes called a \emph{normal Cayley
graph}, see for example \cite{hhc,llt}. However, the reader is warned that
more recently this term has been used in a completely different sense: a
Cayley graph $\Gamma=\Cay(G,S)$ is normal if the group of right translations of
$\Gamma$ is a normal subgroup of $\Aut(\Gamma)$, see for example \cite{xu}.

To avoid confusion, I propose to call a normal Cayley graph in the first sense
above an \emph{inner-automorphic} Cayley graph. Note that this condition is
equivalent to saying that the graph is invariant under both left and right
translation. (For the composition of right translation by $g$ and left
translation by $g^{-1}$ is conjugation by $g$, and so a graph invariant under
two of these maps is invariant under the third also.)

\begin{prop}
The Cayley graph $\Cay(G,S)$ is inner-automorphic if and only if $S$ is a
union of conjugacy classes in $G$.
\end{prop}

Note also that the minimal (non-null) inner-automorphic Cayley graphs for
$G$ are the relations of the \emph{conjugacy class association scheme} on $G$,
see~\cite{cap,godsil,woldar} (the last of these references calls this structure
the \emph{group scheme} of $G$).

I shall call the Cayley graph $\Cay(G,S)$ \emph{automorphic} if it is
invariant under the whole of $\Aut(G)$, that is, if $S$ is a union of orbits
of $\Aut(G)$ acting on $G$. (Thus, if $\Cay(G,S)$ is automorphic, then its
automorphism group contains the \emph{holomorph} of $G$.) These graphs would
fall under the rubric  considered here, although I shall not be discussing
them further.

\section{Dramatis personae}

This section introduces the specific graphs on a group that I will be mainly
concerned with.

\subsection{The commuting graph}

Let $G$ be a finite group. The \emph{commuting graph} of $G$ is the graph
with vertex set $G$ in which two vertices $x$ and $y$ are joined if $xy=yx$.
This graph was introduced by Brauer and Fowler in their seminal paper~\cite{bf}
showing that only finitely many groups of even order can have a prescribed
centraliser. (Brauer and Fowler do not use the word ``graph'', but define
the graph metric in the induced subgraph of the commuting graph on
$G\setminus\{1\}$ and use this. The argument begins by showing that, if there
are at least two conjugacy classes of involutions, then any two involutions
have distance at most~$3$ by a path in the commuting graph which avoids the
identity.)

The commuting graph has had further applications in group theory. Vertices in
$Z(G)$ are joined to everything, and for investigating questions such as
connectedness these are often removed; this makes no difference here.

Also the definition puts a loop at every vertex. There is good reason for 
doing this. It follows from results of Jerrum~\cite{jerrum} on the
``Burnside process'' that the limiting distribution of the random walk on the
commuting graph with loops is uniform on conjugacy classes -- that is, the
limiting probability of being at a vertex is inversely proportional to the
size of its conjugacy class. This is useful in finding representatives of
very small conjugacy classes in large groups.

Closely related is the fact that the \emph{commuting ratio} of a group $G$,
the probability that two randomly-chosen elements commute, is the ratio of the
number of ordered edges of the commuting graph (including loops) to $|G|^2$:
see for example \cite{guro,eberhard}.

But for my purposes here, I will imagine that the loops have been silently
removed.

As with all of these graphs, we can ask: Which groups are characterised by
their commuting graphs? For example, abelian groups of the same order have
isomorphic commuting graphs, as do the dihedral and quaternion groups of
order~$8$. Figure~\ref{f:dq} the commuting graphs of the two groups
$D_8=\langle a,b:a^4=1, b^2=1, b^{-1}ab=a^{-1}\rangle$ and
$Q_8=\langle a,b:a^4=1, b^2=a^2, b^{-1}ab=a^{-1}\rangle$.

\begin{figure}[htb]
\begin{center}
\setlength{\unitlength}{1mm}
\begin{picture}(20,35)
\multiput(0,10)(0,10){2}{\circle*{2}}
\multiput(10,10)(0,10){2}{\circle*{2}}
\multiput(0,10)(10,0){2}{\line(0,1){10}}
\multiput(0,10)(0,10){2}{\line(1,0){10}}
\put(0,10){\line(1,1){10}}
\put(0,20){\line(1,-1){10}}
\multiput(10,10)(0,10){2}{\line(4,-5){8}}
\multiput(18,0)(0,10){2}{\circle*{2}}
\put(18,0){\line(0,1){10}}
\put(10,10){\line(1,0){8}}
\put(10,20){\line(2,-5){8}}
\multiput(10,10)(0,10){2}{\line(4,5){8}}
\multiput(18,20)(0,10){2}{\circle*{2}}
\put(10,20){\line(1,0){8}}
\put(10,10){\line(2,5){8}}
\put(18,20){\line(0,1){10}}
\put(-4,9){$a^3$}
\put(-4,19){$a$}
\put(8,6){$1$}
\put(8,22){$a^2$}
\put(20,-1){$b$}
\put(20,9){$a^2b$}
\put(20,19){$ab$}
\put(20,29){$a^3b$}
\end{picture}
\end{center}
\caption{\label{f:dq}Commuting graph of $D_8$ or $Q_8$}
\end{figure}
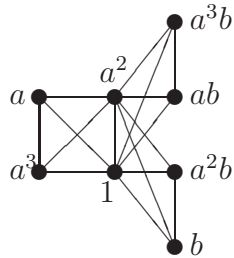

It was conjectured in \cite{aam} and proved in \cite{as,hcg,sw} that any
non-abelian finite simple group is characterised by its commuting graph. 

I note in passing an application of the commuting graphs of finite groups to
the structure of finite quotients of the multiplicative group of a
finite-dimensional division algebra by Segev~\cite{segev}.

To conclude this section, a simple observation about the commuting graph:

\begin{prop}\label{p:maxclcom}
A maximal clique in the commuting graph of $G$ is a maximal abelian subgroup
of~$G$.
\end{prop}

\subsection{The deep commuting graph}

The \emph{deep commuting graph} of a finite group $G$ was introduced very
recently~\cite{ck}. Two elements of $G$ are joined in the deep commuting
graph if and only if their preimages in every central extension of $G$
(that is, every group $H$ with a central subgroup $Z$ such that $H/Z\cong G$)
commute. More specifically, take the commuting graph of a \emph{Schur cover}
\cite{schur} of $G$ (this is a central extension $H$ of largest order such that
$Z$ is contained in the derived group of $H$), and take the induced subgraph of
the commuting graph of $H$ on a transversal to $Z$. It can be shown that the
resulting graph is independent of the choice of Schur cover.

For example, $D_8$ and $Q_8$ are Schur covers of the Klein group $V_4$;
Figure~\ref{f:dq} shows that the deep commuting graph of the Klein group is the star $K_{1,3}$, though its commuting graph is the complete graph $K_4$.

We note in particular that the deep commuting graph is equal to the commuting
graph if the \emph{Schur multiplier} of $G$ (the central subgroup $Z$ in a
Schur cover) is trivial. The converse is false, as we will see in
Proposition~\ref{p:dc}.

\begin{question}
Is it true that a non-abelian finite simple group is characterised by its
deep commuting graph?
\end{question}

\subsection{The power graph}

The \emph{directed power graph} of $G$ is the directed graph with vertex set
$G$, with an arc $x\to y$ if $y=x^m$ for some integer $m$. The \emph{power
graph} of $G$ is the graph obtained by ignoring directions and double arcs;
in other words, $x$ is joined to $y$ if one of $x$ and $y$ is a power of the
other. It is clearly a spanning subgraph of the commuting graph. The power
graph was introduced by Kelarev and Quinn~\cite{kq}.

The directed power graph is a \emph{partial preorder}, that is, a reflexive
and transitive relation on $G$; and the power graph is its
\emph{comparability graph} (two vertices joined if and only if they are related
in the preorder). Comparability graphs of partial preorders and
partial orders form the same class. For clearly every partial order is a
partial preorder. For the converse, if $(A,{\leqslant})$ is a partial preorder,
then the relation $\equiv$ defined by $a\equiv b$ if $a\leqslant b\leqslant a$
is an equivalence relation; putting a total order on each equivalence class
gives a partial order with the same comparability graph as $(A,{\leqslant})$.
It follows from Dilworth's theorem that the power graph of a finite group is
perfect. This was first proved by Feng \emph{et~al.}~\cite{fmw}; see also
\cite{aetal}.

The power graph does not uniquely determine the directed power graph; for
example, if $G$ is the cyclic group of order~$6$, then the identity and the
two generators are indistinguishable in the power graph (they are joined to
all other vertices), but one is a sink and the other two are sources in the
directed power graph. However, the following is shown in~\cite{cameron}:

\begin{theorem}
If two finite groups have isomorphic power graphs, then they have isomorphic
directed power graphs.
\end{theorem}

This is false for infinite groups; see~\cite{cg}.

For a more extensive survey of power graphs, I refer to \cite{akc}.

\subsection{The enhanced power graph}

In 2007, Abdollahi and Hassanabadi~\cite{ah1,ah2} studied a graph they called
the \emph{noncyclic graph} of a group $G$, in which two vertices $x$ and $y$
are joined if $\langle x,y\rangle$ is not cyclic. (For technical reasons 
they excluded the set of isolated vertices, the so-called
\emph{cyclicizer} of $G$.) Later in this paper, some of their results will be
discussed.

For comparison with the other graphs considered in this paper, I will take
the complement of their graph, which was independently defined in the
paper~\cite{aetal} under the name \emph{enhanced power graph} of $G$.
Also, I will not initially assume that vertices joined to all others are
excluded; we will examine such vertices later.

Thus, the enhanced power graph of a group $G$ has vertex set $G$,
with $x$ and $y$ joined if and only if $\langle x,y\rangle$ is cyclic.
Equivalently, $x$ and $y$ are joined if there is an element $z\in G$ such that
each of $x$ and $y$ is a power of $z$. This graph was introduced to interpolate
between the power graph and the commuting graph, but has now been studied in
its own right, especially by Samir Zahirovi\'c and coauthors~\cite{zahir,zbm},
and is in some respects easier to handle than the power graph.

The enhanced power graph can be obtained from the directed power graph by
joining two vertices if both lie in the closed out-neighbourhood of some vertex.
Thus, if two groups have isomorphic power graphs, then they have isomorphic
enhanced power graphs. The converse is also true, see~\cite{zbm}:

\begin{theorem}
For a pair of finite groups, the following are equivalent:
\begin{enumerate}\itemsep0pt
\item the power graphs are isomorphic;
\item the directed power graphs are isomorphic;
\item the enhanced power graphs are isomorphic.
\end{enumerate}
\end{theorem}

It is, however, not true that the power graph, directed power graph, and
enhanced power graph of $G$ have the same automorphism group. For $G=C_6$,
all three automorphism groups are different.

\begin{question}
Is there a simple algorithm for constructing the directed power graph or the
enhanced power graph from the power graph, or the directed power graph from
the enhanced power graph?
\end{question}

Note that the enhanced power graph of $G$ is a union of complete subgraphs on
the maximal cyclic subgroups of $G$. Similarly, the commuting graph is a union
of complete subgraphs on the maximal abelian subgroups. In fact, more is true:

\begin{prop}\label{p:maxclepow}
A maximal clique in the enhanced power graph of $G$ is a maximal cyclic subgroup
of $G$.
\end{prop}

This follows from the fact that if a set of elements of $G$ have the property
that any two generate a cyclic subgroup, then the whole set generates a cyclic
subgroup: see \cite[Lemma 32]{aetal}.

\medskip

At this point I mention another graph which I will not discuss in so much
detail. It has been studied under the name \emph{intersection graph}
(see for example~\cite{cs,hi}), but
I wish to reserve this term for a different concept. Since, as we will see,
it is dual (in a certain vague sense) to the enhanced power graph, I will call
it the \emph{dual enhanced power graph} of $G$, and denote it $\DEP(G)$.
Since the identity is always isolated in this graph (unlike the other graphs
discussed so far), it is natural to remove it and define the vertex set of
$\DEP(G)$ to be $G\setminus\{1\}$.

Two non-identity elements $x$ and $y$ are joined in the dual enhanced power
graph of $G$ if $\langle x\rangle\cap\langle y\rangle$ is not the identity.
Recall that two vertices $x,y$ are joined in the enhanced power graph if they
have a common in-neighbour in the directed power graph; we see that two
vertices $x$ and $y$ are joined in $\DEP(G)$ if they have a common
out-neighbour different from the identity.

In particular, we see that the directed power graph determines the dual
enhanced power graph. But, unlike the case with the power graph and enhanced
graph, it does not work in the reverse direction. If $G$ and $H$ are the
cyclic group and quaternion group of order $8$, then $\DEP(G)$ and $\DEP(H)$
are complete graphs on $7$ vertices, but the directed power graphs of these
two groups are not isomorphic.

\subsection{The generating graph}

The \emph{generating graph} of a finite group $G$ has vertex set
$G$, with $x$ and $y$ joined if and only if $\langle x,y\rangle=G$. If the
minimum number of generators of $G$ is greater than~$2$, then the generating
graph is the null graph. If $G$ is cyclic, then its generating graph has
loops; we will not be too much interested in this case. Note that, by the
Classification of Finite Simple Groups, every non-abelian finite simple group
is $2$-generated. 

The generating graph was introduced in~\cite{ls}, and studied further 
in~\cite{bgk}. The results are often phrased in terms of the \emph{spread},
a graph-theoretic parameter defined as follows: $\Gamma$ has spread (at least)
$k$ if every set of $k$ vertices has a common neighbour. Thus a graph has
spread $1$ if it has no isolated vertex, while spread~$2$ is stronger
than having diameter at most $2$.

In \cite{bgk} it was shown that the generating graph of a non-abelian finite
simple group has positive \emph{spread}. A substantial strengthening has
recently been proved by Burness~\emph{et al.}~\cite{bgh}:

\begin{theorem}\label{t:bgh}
For a finite group $G$ with reduced generating graph $\Gamma$, the following
three conditions are equivalent:
\begin{enumerate}\itemsep0pt
\item any non-identity vertex has a neighbour in the generating graph (that
is, $\Gamma$ has spread~$1$);
\item any two non-identity vertices have a common neighbour in the generating
graph (that is, $\Gamma$ has spread~$2$);
\item any proper quotient of $G$ is cyclic.
\end{enumerate}
\end{theorem}

In particular, the conditions hold for a non-abelian finite simple group. 
However, it is false for infinite groups: even (a) can fail for $2$-generator
infinite groups with all proper quotients cyclic~\cite{cox}.

In the sequel, I will often consider the \emph{non-generating graph},
the complement of the generating graph.

\subsection{The hierarchy}

These graphs are given \emph{ad hoc} names in the literature, but since I will
be talking about all of them here, I prefer to give them names which help to
distinguish them. Thus, the commuting graph of $G$ will be $\Com(G)$; the
deep commuting graph $\DCom(G)$; the power graph $\Pow(G)$; the directed power
graph $\DPow(G)$; the enhanced power graph $\EPow(G)$; the generating graph
$\Gen(G)$; and the non-generating graph $\NGen(G)$. In each case, the vertex
set is the group $G$. Sometimes I refer to the \emph{reduced graph} of one of
the above types, and denote it by a superscript $-$; this usually means that
the identity element is deleted from the vertex set.

There are inclusions between these graphs, as follows. Here $E(\Gamma)$ denotes
the edge set of a graph $\Gamma$; thus $E(\Gamma_1)\subseteq E(\Gamma_2)$
means that $\Gamma_1$ is a \emph{spanning subgraph} of $\Gamma_2$ (a subgraph
using all of the vertices and some of the edges).

\begin{prop}
Let $G$ be a finite group.
\begin{enumerate}\itemsep0pt
\item $E(\Pow(G))\subseteq E(\EPow(G))\subseteq E(\DCom(G))\subseteq 
E(\Com(G))$.
\item If $G$ is non-abelian or not 2-generated, then $E(\Com(G)\subseteq
E(\NGen(G))$.
\end{enumerate}
\end{prop}

\begin{proof} (a) All is obvious except possibly the inclusion of $E(\EPow(G))$
in $E(\DCom(G))$. So suppose that $\langle x,y\rangle$ is a cyclic subgroup
of $G$, and let $H$ be a central extension of $G$, with $H/Z\cong G$. The
lift of $\langle x,y\rangle$ is the extension of a central subgroup $Z$ by
a cyclic group, and hence is abelian; so the lifts of $x$ and $y$ commute
in $H$.

\medskip
(b) If $G$ is not $2$-generated, then $\NGen(G)$ is the complete graph, and
the result is clear. If $G$ is non-abelian, it cannot be generated by two
commuting elements.  \qed
\end{proof}

Because of this, I will refer to the null graph, power graph, enhanced power
graph, deep commuting graph, commuting graph, non-generating graph (in the 
case that $G$ is non-abelian) and complete graph on $G$ as the \emph{graph
hierarchy}, or just \emph{hierarchy}, of $G$.

\medskip

To conclude this section, I note that the dual enhanced power graph does not
fit very well into the hierarchy. Recall that a superscript $-$ means that
the identity is removed; the identity is isolated in $\DEP(G)$ so it is
natural to remove it from this graph also.

\begin{prop}
\begin{enumerate}\itemsep0pt
\item For any finite group $G$, $E(\Pow^-(G))\subseteq E(\DEP(G))$.
\item If $Z(G)=1$, then $E(\DEP(G))\subseteq E(\NGen^-(G))$.
\item In general $\DEP(G)$ is incomparable with the other graphs in the
hierarchy.
\end{enumerate}
\end{prop}

\begin{proof}
(a) Suppose that $\{x,y\}$ is an edge of $\Pow(G)$ with $x,y\ne1$. Without
loss of generality, $x$ is a power of $y$; then
$\langle x\rangle\cap\langle y\rangle=\langle x\rangle$.

\smallskip

(b) Suppose that $Z(G)=1$, and that $\{x,y\}$ is an edge of $\DEP(G)$. Let
$\langle x\rangle\cap\langle y\rangle=\langle z\rangle$, with $z\ne1$. Then
$x,y\in C_G(z)<G$, so $\langle x,y\rangle\ne G$.

\smallskip

(c) If $G=C_6$, then $\EPow^-(G)$, and all higher graphs in the hierarchy,
are complete, but $\DEP(G)$ is not. If $G=Q_8$, then $\Com^-(G)$, and all lower
graphs in the hierarchy, are incomplete, but $\DEP(G)$ is complete. \qed
\end{proof}

\begin{question}
For which groups $G$, and for which types $\mathrm{X}$ of graph in the
hierarchy, does $\DEP(G)=\mathrm{X}^-(G)$ hold?
\end{question}

\subsection{The Gruenberg--Kegel graph}

A related graph will play a role in the investigation in several places.
The \emph{Gruenberg--Kegel graph}, also known as the \emph{prime graph},
of a finite group $G$ has vertex set the set of prime divisors of the order
of $G$; vertices $p$ and $q$ are joined by an edge if and only if $G$ contains
an element of order $pq$. 

The graph was introduced in an unpublished manuscript by Gruenberg and Kegel
to study the integral group ring of a finite group, and in particular the
decomposability of the augmentation ideal: see~\cite{gr}. The main structural
result was published by Williams (a student of Gruenberg) in~\cite{williams}.
It asserts that groups whose Gruenberg--Kegel graph is disconnected have a
very restricted structure.

\begin{theorem}\label{t:gk}
Let $G$ be a finite group whose Gruenberg--Kegel graph is disconnected. Then
one of the following holds:
\begin{enumerate}\itemsep0pt
\item $G$ is a Frobenius or $2$-Frobenius group; 
\item $G$ is an extension of a nilpotent $\pi$-group by a simple group by a
$\pi$-group, where $\pi$ is the set of primes in the connected component
containing $2$.
\end{enumerate}
\end{theorem}

A \emph{$2$-Frobenius group} is a group $G$ with normal subgroups $H$ and $K$
with $H\le K$ such that
\begin{itemize}\itemsep0pt
\item $K$ is a Frobenius group with Frobenius kernel $H$;
\item $G/H$ is a Frobenius group with Frobenius kernel $K/H$.
\end{itemize}
A typical example is the group $G=S_4$, with $K=A_4$, $H=V_4$ (the Klein group),
and $G/K\cong S_3$.

Williams went on to examine the known finite simple groups to determine which
ones could occur in conclusion~(b) of the Theorem. He could not handle the 
groups of Lie type in characteristic~$2$; this was completed by Kondrat'ev in
1989, and some errors corrected by Kondrat'ev and Mazurov in 2000.

\medskip

The next result indicates that the Gruenberg--Kegel graph is closely connected
with our hierarchy of graphs.

\begin{theorem}
Let $G_1$ and $G_2$ be groups whose power graphs, or enhanced power graphs, or
deep commuting graphs, or commuting graphs, are isomorphic. Then the 
Gruenberg--Kegel graphs of $G_1$ and $G_2$ are equal.
\end{theorem}

\begin{proof}
The four possible hypotheses each imply that $G_1$ and $G_2$ have the same
order, so their GK graphs have the same set of vertices.

We show that in all cases except the power graph, primes $p$ and $q$ are
adjacent in the GK graph of $G$ if and only if there is a maximal clique in the
graph on $G$ with size divisible by $pq$. This is clear in the cases of the
enhanced power graph and the commuting graph; for, as we observed earlier,
the maximal cliques in these are maximal cyclic subgroups or maximal abelian
subgroups of $G$ respectively (Propositions~\ref{p:maxclcom}
and~\ref{p:maxclepow}), and if their order is divisible by $pq$ then
they contain elements of order $pq$. Conversely an element of order $pq$ is
contained in a maximal cyclic (or abelian) subgroup.

Consider the deep commuting graph of a group $G$. Let $H$ be a Schur cover
of $G$, with $H/Z\cong G$. A maximal clique has the form $A=B/Z$, where $B$
is a maximal abelian subgroup of $H$ (containing $Z$). So $A$ is an abelian
subgroup of $G$, and if $pq$ divides $|A|$ then $A$ contains an element of
order $pq$. Conversely, suppose that $p$ and $q$ are joined in the GK
graph, and let $x$ and $y$ be commuting elements of orders $p$ and $q$ in $G$,
and $a$ and $b$ their lifts in $H$. Then $a$ and $b$ are contained in
$\langle Z,ab\rangle$, which is an extension of a central subgroup by a cyclic
group and hence is abelian; so $a$ and $b$ commute. Choosing a maximal abelian
subgroup of $H$ containing $a$ and $b$ and projecting onto $G$ gives a maximal
clique in $\DCom(G)$ with order divisible by $pq$.

This fails for the power graph. Instead we use the fact that groups with
isomorphic power graphs also have isomorphic enhanced power graphs, and
so have equal GK graphs, by what has already been proved. \qed
\end{proof}

The proof gives a little more. Suppose, for example, that $G$ and $H$ are
groups for which $\Com(G)$ is isomorphic to $\EPow(H)$. Then the
Gruenberg--Kegel graphs of $G$ and $H$ are equal. 

I do not know whether the analogous result holds for the non-generating graph
of a non-abelian $2$-generated group.

\medskip

The Gruenberg--Kegel graph is an active topic of research;
see~\cite{cm} for a survey and some recent results. Some of the research 
concerns the question of whether a group is determined (perhaps up to finitely
many possibilities) by its GK graph. There are two versions of this: two
GK graphs could be equal (as graphs whose vertex set is a finite set of primes)
or merely isomorphic as graphs but with possibly different labels for the
vertices. For an interesting example, the GK graphs of the groups $A_{10}$ 
and $\Aut(\mathrm{J}_2)$ are isomorphic, and both have vertex sets
$\{2,3,5,7\}$, but are not equal: the labels $2$ and $3$ are swapped.

\subsection{Intersection graphs}

Let $G$ be a finite group, not trivial and not a cyclic group of prime order.
The \emph{intersection graph} of $G$ is the graph whose vertices are the
non-trivial proper subgroups of $G$, with two vertices $H_1$ and $H_2$ adjacent
if $H_1\cap H_2\ne\{1\}$.

There are various other intersection graphs: we can restrict to subgroups in
a particular class, or to maximal subgroups.

We will see a connection between some intersection graphs and some of the
graphs in our hierarchy.

\section{Equality and differences}

For a non-abelian finite group $G$, there are seven graphs in the hierarchy, 
and a natural question is: When can two of them be equal? If they are not
equal, what can be said about their difference?

\subsection{Equality}

At the two ends, things are easy:
\begin{prop}
\begin{enumerate}\itemsep=0pt
\item $\Pow(G)$ is equal to the null graph if and only if $G$ is the trivial
group.
\item $\NGen(G)$ is equal to the complete graph if and only if $G$ is not
$2$-generated.
\item $\NGen(G)=\Com(G)$ if and only if $G$ is either abelian and not
$2$-generated, or a minimal non-abelian group.
\end{enumerate}
\end{prop}

\begin{proof}
Parts (a) and (b) are clear. So suppose that $\NGen(G)=\Com(G)$ and this is
not the complete graph. Then $G$ is non-abelian, but if $x$ and $y$ do not
generate $G$ then they commute; so every proper subgroup of $G$ is abelian.
Thus $G$ is minimal non-abelian. \qed
\end{proof}

The minimal non-abelian groups were determined by Miller and Moreno~\cite{mm}
in 1903. There are two types: the first consists of groups of prime power
order; the second are extensions of an elementary abelian $p$-group by a 
cyclic $q$-group, where $p$ and $q$ are primes.

\medskip

Leaving aside the deep commuting graph from the present, the following was
shown in \cite{aetal}:

\begin{prop}
Let $G$ be a finite group.
\begin{enumerate}\itemsep0pt
\item The power graph of $G$ is equal to the enhanced power graph if and only
if $G$ contains no subgroup isomorphic to $C_p\times C_q$, where $p$ and $q$
are distinct primes; equivalently, the Gruenberg--Kegel graph of $G$ is a
null graph.
\item The enhanced power graph of $G$ is equal to the commuting graph if and
only if $G$ contains no subgroup isomorphic to $C_p\times C_p$, where $p$
is prime; equivalently, the Sylow $p$-subgroups of $G$ are cyclic or 
generalized quaternion groups.
\item The power graph of $G$ is equal to the commuting graph if and only if
$G$ contains no subgroup isomorphic to $C_p\times C_q$, where $p$ and $q$
are primes (equal or distinct).
\end{enumerate}
\end{prop}

\begin{proof}
(a) If $G$ contains commuting elements of orders $p$ and $q$, they are
adjacent in $\EPow(G)$ but not in $\Pow(G)$. Conversely, suppose that $x$
and $y$ are adjacent in $\EPow(G)$ but not in $\Pow(G)$. Then $x$ and $y$
are contained in a cyclic group $C$ but neither is a power of each other;
$C$ must then have order divisible by two distinct primes.

\medskip

(b) If $G$ contains commuting elements of the same prime order $p$ but not in a
cyclic subgroup of order~$p$, they are joined in the commuting graph but not
in the enhanced power graph. Conversely, suppose that $x$ and $y$ are adjacent
in $\Com(G)$ but not in $\EPow(G)$. The orders of $x$ and $y$ must have a
common factor (otherwise they generate a cyclic group); so some powers of
them have prime order $p$ and generate $C_p\times C_p$.

Now a theorem of Burnside (see~\cite[Theorem 12.5.2]{hall}) shows that a
$p$-group containing no subgroup $C_p\times C_p$ is cyclic or generalized
quaternion. 

\medskip

(c) The third part is immediate from the first two. \qed
\end{proof}

Using these results it is possible to classify the groups involved.

\medskip

\begin{enumerate}\itemsep0pt
\item Any group of prime power order has Gruenberg--Kegel graph consisting
of a single vertex, so has power graph equal to enhanced power graph. Any
other group with this property has disconnected Gruenberg--Kegel graph,
and so satisfies the conclusion of Theorem~\ref{t:gk}. These groups are the
ones with the property that every element has prime power order; they are
sometimecs called \emph{EPPO groups}. They, and some generalisations, were
studied intensively in the 1960s and 1970s by Graham Higman and his students
in Oxford (see for example~\cite{higman}). Recently, Cameron and Maslova
\cite{cm} have given a complete list of EPPO groups.
\item A group with all Sylow subgroups cyclic is metacyclic; indeed, if
the primes dividing its order are $p_1,p_2,\ldots,p_r$ in increasing order,
then it has a normal Hall subgroup corresponding to the last $i$ primes in
this list, for $1\le i\le r-1$.

By Glauberman's $Z^*$-theorem~\cite{glaub}, if $G$ has generalized quaternion
Sylow $2$-subgroup, and $O(G)$ is the largest normal subgroup of odd order in
$G$, then $G/O(G)$ has a unique involution; the quotient $\bar G$ by the
subgroup  generated by this involution has dihedral Sylow $2$-subgroup, so
falls into the classification by Gorenstein and Walter~\cite{gw}. Of the groups
in their theorem, we retain only those with cyclic Sylow subgroups for odd
primes, that is, $\bar{G}$ is isomorphic to $\mathrm{PSL}(2,p)$ or
$\mathrm{PGL}(2,p)$ or to a dihedral $2$-group. Conversely, each such group
can be lifted to a unique group with a unique involution. The normal subgroup
$O(G)$ has all its subgroups cyclic, so is metacyclic, as above.
\end{enumerate}

Finally, the deep commuting graph lies between the enhanced power graph and
the commuting graph. In order to investigate equality here, we need another
construction. Recall that the Schur multiplier of $G$ is the largest kernel
$Z$ in a stem extension $H$ of $G$ (with $Z\le Z(H)\cap H'$ and $H/Z\cong G$).
An extension is said to be \emph{commutation-preserving}, or CP, if whenever
two elements $x,y\in G$ commute, their preimages in $H$ also commute. Now
there is a well-defined largest kernel of a CP stem extension of $G$; this
is the \emph{Bogomolov multiplier} of $G$, see \cite{bog,jm}.

The Bogomolov multiplier first arose in connection with the work of Artin and
Mumford on obstructions to Noether’s conjecture on the pure transcendence of
the field of invariants; it is also connected with other topics in number
theory such as the Tate--Shafarevich set, and in group theory such as the
coclass. However, we only need the definition given above.

\begin{prop}\label{p:dc}
Let $G$ be a finite group.
\begin{enumerate}\itemsep0pt
\item $\DCom(G)=\EPow(G)$ if and only if $G$ has the following property: let $H$
be a Schur cover of $G$, with $H/Z=G$. Then for any subgroup $A$ of $G$, with
$B$ the corresponding subgroup of $H$ (so $Z\leqslant B$ and $B/Z=A$), if $B$
is abelian, then $A$ is cyclic.
\item $\DCom(G)=\Com(G)$ if and only if the Bogomolov multiplier of $G$ is
equal to the Schur multiplier.
\end{enumerate}
\end{prop}

I refer to \cite{ck} for the proofs.

\medskip

A precise characterisation of the groups attaining either equality is not
known; but examples exist where one bound but not the other is met, or
where neither bound is met (see~\cite{ck}):
\begin{itemize}\itemsep0pt
\item If $G$ is the symmetric or alternating group of degree at least $8$,
then $E(\EPow(G))\subset E(\DCom(G))\subset E(\Com(G))$.
\item If $G$ is a dihedral group of order $2^n$ with $n\ge3$, then 
$E(\EPow(G))=E(\DCom(G))\subset E(\Com(G))$.
\item If $G$ is a certain group of order $64$ (number $182$ in the \textsf{GAP}
library), then $E(\EPow(G))\subset E(\DCom(G))=E(\Com(G))$.
\end{itemize}
Note that
\begin{itemize}\itemsep0pt
\item if the Schur multiplier of $G$ is trivial, then $\DCom(G)=\Com(G)$;
\item in general, the Bogomolov multiplier is much smaller than the Schur
multiplier; for example, if $G$ is a non-abelian finite simple group, then
its Bogomolov multiplier is trivial~\cite{kun}.
\end{itemize}

\begin{question}
\begin{enumerate}\itemsep0pt
\item What can be said about groups $G$ for which $\DCom(G)=\EPow(G)$?
\item What can be said about groups $G$ for which $\DCom(G)=\Com(G)$?
\end{enumerate}
\end{question}

\subsection{Differences}

For any pair of graphs in the hierarchy, if $G$ is a group such that these
two graphs are unequal, we could ask about the graph whose edge set is the
difference. We could denote these by using, for example, $(\Com{-}\Pow)(G)$
for the graph whose edges are those belonging to the commuting graph but not
the power graph, with similar notation in other cases.

At the top, the difference between the complete graph and the non-generating
graph is just the generating graph, which has been extensively studied. At
the next level, the difference between the generating graph and the
commuting graph (the graph $(\NGen{-}\Com)(G)$) has been studied by Saul
Freedman; the results will appear in his thesis. The most complete results
are for nilpotent groups, and are reported in \cite{cfrd}. In particular, if
$G$ is nilpotent and the non-commuting non-generating graph is not null, then 
after deletion of all isolated vertices it is connected, with diameter $2$
or~$3$.

Other differences (apart from the difference between the power graph and the
null graph) have not been studied.

\begin{question}
For each pair of graph types in the hierarchy, what can be said about groups
for which the difference is connected (after removing isolated vertices and
vertices joined to all others)?
\end{question}

In section~\ref{s:conn}, I give a very weak partial result on the graph
$(\Com{-}\Pow)(G)$.

\subsection{Further problems}

We saw the result of \cite{zbm} that two groups have isomorphic power graphs
if and only if they have isomorphic enhanced power graphs.

\begin{question}
Are there any other implications of this kind between pairs of graphs in the
hierarchy?
\end{question}

For a simple negative example, the groups $C_{p^2}$ and $C_p\times C_p$ have
isomorphic commuting graphs but nonisomorphic power graphs, while the group
$C_p\times C_p\times C_p$ and the non-abelian group of order $p^3$ and
exponent $p$ have isomorphic power graphs but nonisomorphic commuting graphs.

\medskip

Do there exist groups $G_1$ and $G_2$ such that, for example, $\Pow(G_1)$
is isomorphic to $\Com(G_2)$? This will be true if $\Pow(G_1)=\Com(G_1)$
and $G_1$ and $G_2$ have isomorphic commuting graphs, or if
$\Pow(G_2)=\Com(G_2)$ and $G_1$ and $G_2$ have isomorphic power graphs.

\begin{question}
Can $\Pow(G_1)$ and $\Com(G_2)$ be isomorphic for groups $G_1$ and $G_2$ which
both have power graph not equal to commuting graph? Similar questions for other
pairs of graphs in the hierarchy.
\end{question}

\section{Cliques and cocliques}
\label{s:cc}

It is clear that, if $\Gamma_1$ and $\Gamma_2$ share a vertex set and 
$E(\Gamma_1)\subseteq E(\Gamma_2)$, then
$\omega(\Gamma_1)\leqslant\omega(\Gamma_2)$,
$\chi(\Gamma_1)\leqslant\chi(\Gamma_2)$,
$\alpha(\Gamma_1)\geqslant\alpha(\Gamma_2)$, and
$\theta(\Gamma_1)\geqslant\theta(\Gamma_2)$. So these four parameters are
monotonic for the graphs in the hierarchy on a given group (non-decreasing
for $\omega$ and $\chi$, non-increasing for $\alpha$ and $\theta$).

\begin{prop}
\begin{enumerate}\itemsep0pt
\item $\omega(\EPow(G))$ is the order of the largest cyclic subgroup of $G$;
\item $\omega(\Com(G))$ is the order of the largest abelian subgroup of $G$;
\item $\omega(\DCom(G))$ is the order of the largest subgroup of $G$ whose
inverse image in any central extension of $G$ is abelian.
\end{enumerate}
\end{prop}

\begin{proof}
(a) and (b) follow from Propositions~\ref{p:maxclcom} and~\ref{p:maxclepow}.
For (c), apply (b) to a Schur cover of $G$. \qed
\end{proof}

The power graph of a group $G$ is perfect, and so has equal clique number and
chromatic number. These numbers do not exceed the clique number of the 
enhanced power graph, which is the largest order of an element of $G$; but
they may be smaller. For example, $\omega(\Pow(C_6))=5$, but $\EPow(C_6)$ is
the complete graph $K_6$.

As noted, $\omega(\Pow(G))\leqslant\omega(\EPow(G))$. There is an inequality
in the other direction:

\begin{theorem}\label{t:bound}
There is a function $f$ on the natural numbers such that, for any finite group
$G$, $\omega(\EPow(G))\leqslant f(\omega(\Pow(G)))$.
\end{theorem}

\begin{proof}
Let $m$ be the largest prime power which is the order of an element of $G$.
The power graph of an $m$-cycle is complete; so $\omega(\Pow(G))\geqslant m$.
On the other hand, the largest order of an element of $G$ is not greater
than the least common multiple of $\{1,\ldots,m\}$, say $f(m)$; and
so $\omega(\EPow(G))\leqslant f(m)$. \qed
\end{proof}

No such result holds for the commuting graph. If $G$ is an elementary abelian
group of order $2^n$, then $\omega(\EPow(G))=2$ but $\omega(\Com(G))=2^n$.

The function $f$ in the theorem is exponential. For there are 
$\pi(m)\leqslant(1+o(1))m/\log m$ primes up to $m$, and $f(m)$ is the
product of the largest power of each prime not exceeding $m$; so
$f(m)\leqslant m^{(1+o(1))m/\log m}=\mathrm{e}^{(1+o(1))m}$. But the true
value is probably much smaller.

\begin{question}
Find the best possible function $f$ in Theorem~\ref{t:bound}.
\end{question}

An upper bound for the chromatic number of the enhanced power graph is given
in \cite[Theorem 12]{aetal}:

\begin{theorem}
Let $G$ be a finite group, and $S$ the set of orders of elements of $G$. Then
\[\chi(\EPow(G))\le\sum_{n\in S}\phi(n),\]
where $\phi$ is Euler's function.
\end{theorem}

\begin{proof} For each $n\in S$, the set of elements of order $n$ in $G$ is
the disjoint union of complete graphs of size $\phi(n)$, and can be coloured
with $\phi(n)$ colours. If we use disjoint sets of colours for different
orders, no further clashes will occur. \qed
\end{proof}

The bound is met for abelian groups. For if $G$ is abelian, then $S$ is the
set of divisors of the exponent of $G$, and so the sum on the right is the
exponent, which is the largest element order in $G$.

\begin{question}
Find a formula for the clique number of the power graph, or the chromatic
number of the enhanced power graph, of a finite group.
\end{question}

Results about the independence number and clique cover number are less well
developed. Since the power graph is perfect, the Weak Perfect Graph Theorem
of Lov\'asz~\cite{lovasz} asserts that its complement is also perfect, so
\[\alpha(\Pow(G))=\theta(\Pow(G))\]
for any finite group $G$. (Alternatively this follows from Dilworth's Theorem.)

The independence number of the non-generating graph of a finite group has been
investigated by Lucchini and Mar\'oti~\cite{lm}.

\section{Induced subgraphs}

In this section I will consider the question, for each of the graphs in our
hierarchy: For which finite graphs $\Gamma$ does there exist a finite group
$G$ such that $\Gamma$ is isomorphic to an induced subgraph of the group of
that type defined on $G$? (An \emph{induced subgraph} of $\Gamma$ on a subset
$A$ of the vertex set consists of the vertices of $A$ and all edges of $\Gamma$
which are contained in $A$.)

To summarise the results:
\begin{itemize}\itemsep0pt
\item A finite graph $\Gamma$ is isomorphic to an induced subgraph of the
power graph of some finite group $G$ if and only if $\Gamma$ is the
comparability graph of a partial order.
\item For each of the other graphs in the hierarchy, every finite graph is
isomorphic to an induced subgraph of that graph defined on some finite group.
\end{itemize}

Three related questions are:

\begin{question}
\begin{enumerate}\itemsep0pt
\item What is the smallest group for which a given graph is embeddable in the
enhanced power graph/deep commuting graph/commuting graph/non-generating graph?
\item What is the smallest group in which every graph on $n$ vertices can be
embedded in one of these graphs?
\item Which graphs occur if we restrict the group to have a particular property
such as nilpotence or simplicity?
\end{enumerate}
\end{question}

Here is a very rough lower bound for the smallest $N$ such that every
$n$-vertex graph can be embedded in the enhanced power graph, deep commuting
graph, commuting graph, or non-generating graph of some group of order at
most $N$. For our rough calculation, we need only consider groups of order $N$.
It is known that there are at most $2^{c(\log N)^3}$ such groups
(see~\cite{bnv}); each has
at most $N^n$ subsets of size $n$. But there are at least $2^{n(n-1)/2}/n!$
graphs on $n$ vertices up to isomorphism. So we require
\[2^{c(\log N)^3}\cdot N^n\ge 2^{n(n-1)/2}/n!,\]
which implies that $N\ge2^{n^{2/3-\epsilon}}$. So the exponential bound we find
in some cases is not too far from the truth.

Note also that every $n$-vertex graph is embeddable in a Paley graph of order
$q$, where $q$ is a prime power congruent to $1$ (mod~$4$) and $q>n^22^{2n-2}$
(see~\cite{beh,bt}); so, to find a group whose commuting graph, etc., embeds
all graphs of order $n$, we only need to embed this Paley graph.

\subsection{The commuting graph}

\begin{theorem}\label{t:ind_com}
Every finite graph is isomorphic to an induced subgraph of the commuting graph
of a finite group. This group can be taken to be nilpotent of class $2$ and
exponent $4$.
\end{theorem}

\begin{proof}
Let $F$ be the two-element field, $V$ a vector space over $F$, and $B$ a 
bilinear form on $V$. Define an operation $\circ$ on $V\times F$ by the rule
\[(v_1,a_1)\circ(v_2,a_2)=(v_1+v_2,a_1+a_2+B(v_1,v_2)).\]
It is a straightforward exercise to show that this operation makes $V\times F$
a group. This group is nilpotent of class $2$ and exponent (dividing) $4$,
since $\{0\}\times F$ is a central subgroup with elementary abelian quotient.
Moreover, $(v_1,a_1)$ and $(v_2,a_2)$ commute if and only if 
$B(v_1,v_2)=B(v_2,v_1)$.

Now a bilinear form is uniquely determined by its values on pairs of vectors
taken from a basis for $V$; these values can be assigned arbitrarily. So
let $\Gamma$ be a graph with vertex set $\{1,\ldots,n\}$, and let
$v_1,\ldots,v_n$ be a basis. Assign the values $B(v_i,v_j)=0$ if $i\leqslant j$;
for $i>j$, put $B(v_i,v_j)=0$ if vertices $i$ and $j$ are adjacent, $1$ if not.
Then it is clear that the induced subgraph of the commuting graph on the
set $\{v_1,\ldots,v_n\}$ is isomorphic to $\Gamma$. \qed
\end{proof}

The construction above shows that the smallest group whose commuting graph
contains a given $n$-vertex graph has order at most $2^{n+1}$ if $\Gamma$ has
$n$ vertices. However, it may be very much smaller; for the complete graph
$K_n$, the answer is clearly $n$.

Using the remark before this subsection, we can build a group whose commuting
graph embeds every $n$-vertex graph, by applying this construction to a
sufficiently large Paley graph.

However, there is a more economical way to proceed, using a
group from a familiar class, the \emph{extraspecial $p$-groups}.
The quaternion group $Q_8$ is generated by elements $a,b,z$ satisfying
$a^2=b^2=[a,b]=z$ and $z^2=1$; its centre is generated by $z$.
Let $G_n$ denote the central product of $n$ copies of $Q_8$. Then $G_n$ is
generated by elements $a_i,b_i$ (for $1\le i\le n$ and $z$, such that, for
each $i$, $a_i,b_i,z$ generate $Q_8$, while $a_i$ and $b_i$ commute with
$a_j$ and $b_j$ for $j\ne i$.

The group $G_n$ has order $2^{2n+1}$.

\begin{theorem}
Every $n$-vertex graph is isomorphic to an induced subgraph of the commuting
graph of $G_n$.
\end{theorem}

\begin{proof}
The proof is by induction on $n$; starting the induction is trivial. So
suppose that the theorem is true for $n$, and let $\Gamma$ be an $(n+1)$-vertex
graph. We can suppose that the induced subgraph on $\{v_1,\ldots,v_n\}$ is
already embedded as an induced subgraph of the group $G_n$.

We modify this embedding as follows. We embed $G_n$ in $G_{n+1}$ in the
obvious way, and replace the vertex $v_i$ by $v_ia_{n+1}$ if $v_i$ is not
joined to $v_{n+1}$, leaving it as it is if these vertices are joined. Since
$a_{n+1}$ commutes with all $v_i$, this does not change the induced subgraph.
Now we embed $v_{n+1}$ as $b_{n+1}$. Clearly this gives an embedding of
$\Gamma$.
\end{proof}

This argument is based on a hint from Persi Diaconis, who has used similar
ideas to study conjugacy classes in the Heisenberg group over a finite field
and generalisations. (Recall that the random walk on the commuting graph has
a limit which is uniform on conjugacy classes.) I do not have further details
at present. Similar ideas can be found in the beautiful paper~\cite{ds}.

\subsection{The deep commuting graph}

\begin{theorem}
Every finite graph is isomorphic to an induced subgraph of the deep commuting
graph of a finite group.
\end{theorem}

\begin{proof}
Let $\Gamma$ be a finite graph. As we have seen, $\Gamma$ is isomorphic to an
induced subgraph of the commuting graph of some group. So it is enough to show
that this group can be chosen to have trivial Schur multiplier. Since the
induced subgraph on a subgroup $H$ of the commuting graph of $G$ is the
commuting graph of $H$, it suffices to show that every finite group can be
embedded in a finite group with trivial Schur multiplier.

By Cayley's Theorem, every finite group of order $n$ can be embedded in the
symmetric group $S_n$. Unfortunately the symmetric group has Schur multiplier
$C_2$ if $n\geqslant8$. So we embed $S_n$ into the general linear group
$\mathrm{GL}(n,2)$ by permutation matrices of order $n$.
Now the Schur multiplier of $\mathrm{GL}(n,2)$ is trivial except for
$n=3$ or $n=4$~\cite{griess}. \qed
\end{proof}

\subsection{The power graph}

We saw that the power graph of $G$ is the comparability graph of a partial
order; hence any induced subgraph is also a comparability graph. The converse
is also true:

\begin{theorem}
A finite graph is isomorphic to an induced subgraph of the power graph of a
finite group if and only if it is the comparability graph of a partial order.
The group can be taken to be cyclic of squarefree order.
\end{theorem}

\begin{proof}
One way round follows from our preliminary remarks: the power graph is the
comparability graph of a partial order, and the class of such graphs is
closed under taking induced subgraphs.

So suppose that we have a
partial order $\leqslant$ on $X$. For each $x\in X$, let
$[x]=\{y\in x:y\leqslant x\}$. A routine check shows that
\begin{itemize}\itemsep0pt
\item $[y]\subseteq[x]$ if and only if $y\leqslant x$;
\item $[x]=[y]$ if and only if $x=y$.
\end{itemize}
So the given partial order is isomorphic to the set of subsets of $X$ of the
form $[x]$, ordered by inclusion.

Now choose distinct prime numbers $p_x$ for $x\in X$. Let $G$ be the direct
product of cyclic groups $C_{p_x}=\langle a_x\rangle$ of order $p_x$ for
$x\in X$. Map the subset $Y$ of $X$ to the element $g_X=(g_x:x\in X)$ of
the direct product, where
\[g_x=\cases{a_x & if $x\in Y$,\cr 1 & otherwise.\cr}\]
It is readily checked that $g_X$ and $g_Y$ are adjacent in the power graph if
and only if $X$ and $Y$ are adjacent in the comparability graph of the inclusion
order on $X$.

To conclude, we note that $G$ is a cyclic group of squarefree order. \qed
\end{proof}

Note that a graph is the comparability graph of a partial order if and only if
there is a transitive orientation of the edges. There is a list of forbidden
induced subgraphs for comparability graphs~\cite{gallai}, but it is not
straightforward to state.

\subsection{The enhanced power graph}

\begin{theorem}
Every finite graph is isomorphic to an induced subgraph of the enhanced power
graph of some group (which can be taken to be abelian).
\end{theorem}

\begin{proof}
The proof is by induction. For a graph with a single vertex, there is no
problem. So let $\Gamma$ be a graph with vertex set $\{1,\ldots,n\}$, and
suppose that $i\mapsto x_i$ (for $i=1,\ldots,n-1$) is an isomorphism to an
induced subgraph of $\EPow(G)$.

Choose a prime $p$ not dividing the order of $G$, and let $H=\langle a,b\rangle$
be an elementary abelian group of order $p^2$. Now in the group $G\times H$,
replace $x_i$ by $x_ia$ if $i$ is not joined to $n$ in $\Gamma$, and leave
it as is if $i$ is joined to $n$. Then map $n$ to $x_n=b$.

Since $p\nmid|G|$, for any $z\in G$ we have
$\langle z,a\rangle=\langle z\rangle\times\langle a\rangle$, which is cyclic.
So the embedding of $\{1,\ldots,n-1\}$ is still an isomorphism to an induced
subgraph. Moreover, $\langle x_i,b\rangle$ is cyclic while
$\langle x_ia,b\rangle$ is not, so we have the correct edges from $b$ to the
other vertices, and the result is proved.

The resulting group is the product of $n$ copies of $C_p\times C_p$ for
distinct primes $p$. \qed
\end{proof}

The dual enhanced power graph is even easier:

\begin{theorem}
Every finite graph is isomorphic to an induced subgraph of the dual enhanced
power graph of some group (which can be taken to be cyclic).
\end{theorem}

\begin{proof}
Let $E$ be the edge set of $\Gamma$; choose distinct primes $p_e$ for each
edge $e\in E$, and let $C_{p_e}=\langle a_e\rangle$ be the cyclic group of
order $p_e$, and $G$ the direct product of all these cyclic groups. Now
represent a vertex $v$ by the element $b_v=\prod_{v\in e}a_e$. Then
\[\langle b_v\rangle\cap\langle b_w\rangle=
\cases{\langle a_e\rangle & if $e=\{v,w\}$,\cr \{1\} & otherwise.\cr}\]
So $\Gamma$ is an induced subgraph of $\DEP(G)$.\qed
\end{proof}

\subsection{The generating graph}
\label{s:indgg}

\begin{theorem}\label{t:gen_univ}
Every finite graph is isomorphic to an induced subgraph of the generating
graph of a finite group.
\end{theorem}

\begin{proof} Let $\Gamma$ be a finite graph. We proceed in a number of
steps.

\subparagraph{Step 1} Replace $\Gamma$ by its complement.

\subparagraph{Step 2} Every graph can be represented as the intersection
graph of a \emph{linear hypergraph}, a family of sets which intersect in at
most one point (where intersection~$1$ corresponds to adjacency). The ground
set $E$ is the set of edges of the graph; the vertex $v$ is represented by
the set $S(v)$ of edges incident with $v$. Then for distinct vertices $v,w$,
\[S(v)\cap S(w)=\cases{e, & if $\{v,w\}$ is an edge $e$,\cr
\emptyset & if $v$ and $w$ are nonadjacent.\cr}\]

\subparagraph{Step 3} Add some dummy points, each lying in just one of the
sets, so that they all have the same cardinality $k$, with $k\ge3$. Now add
some dummy points in none of the sets so that the cardinality $n$ of the set
$\Omega$ of points satisfies the conditions that $n>2k$ and $n-k$ is prime.

\subparagraph{Step 4} Now replace each set by its complement. The complements
of two subsets of $\Omega$ have union $\Omega$ if and only if the two sets
are disjoint. Thus, each original vertex is now represented by an $(n-k)$-set
where two such sets have union $\Omega$ if and only if the corresponding
vertices are adjacent in $\Gamma$.

\subparagraph{Step 5} Replace each set by a cyclic permutation on that set,
fixing the remaining points. Each of these cycles has odd prime length, so each
is an even permutation, and so lies in the alternating group $A_n$. Let
$g_v$ be the permutation corresponding to the vertex $v$ of $\Gamma$.
\begin{itemize}\itemsep0pt
\item If $v$ and $w$ are nonadjacent, then the supports of $g_v$ and $g_w$ have
union strictly smaller than $\Omega$, so $\langle g_v,g_w\rangle\ne A_n$.
\item Suppose $v$ and $w$ are adjacent. Then the supports of $g_v$ and $g_w$
have union $\Omega$, so $H=\langle g_v,g_w\rangle$ is transitive on $\Omega$.
It is primitive: for each of $g_v$ and $g_w$ is a cycle of prime length 
$n-k>n/2$, and a block of imprimitivity either contains the cycle (and so
has length greater than $n/2$, hence $n$) or meets it in one point (and so
there are more than $n/2$ blocks, hence $n$ blocks). Hence $H$ is a primitive
group of degree $n$ containing a cycle of prime length $p$ with $n/2<p<n-2$, By
Jordan's theorem~\cite[Theorem~13.9]{wielandt}, $H$ contains the alternating
group $A_n$. Since it is generated by even permutations, $H=A_n$.
\end{itemize}

Thus we have embedded $\Gamma$ as an induced subgraph in the generating graph
of $A_n$, as required. \qed
\end{proof}

\subsection{Differences}

We can also ask which graphs can be embedded in the graph whose edge set is
the difference of the edge sets of two graphs in the hierarchy.

The proof that enhanced power graphs are universal uses abelian groups for the
embedding. So, by embedding the complement, it shows:

\begin{corollary}
Let $\Gamma$ be a finite graph. Then there is a group $G$ such that $\Gamma$
is isomorphic to an induced subgraph of $\Com{-}\EPow(G)$.
\end{corollary}

However, a much stronger result is true:

\begin{theorem}
Let $\Gamma$ be a finite complete graph, whose edges are coloured red, green
and blue in any manner. Then there is an embedding of $\Gamma$ into a finite
group $G$ so that 
\begin{enumerate}\itemsep0pt
\item vertices joined by red edges are adjacent in the enhanced power graph;
\item vertices joined by green edges are adjacent in the commuting graph but
not in the enhanced power graph;
\item vertices joined by blue edges are non-adjacent in the commuting graph.
\end{enumerate}
\end{theorem}

\begin{proof}
We begin with two observations. First, the direct product of cyclic (resp.\ 
abelian) groups of coprime orders is cyclic (resp.\ abelian).

Second, consider the non-abelian group of order $p^3$ and exponent $p^2$,
where $p$ is an odd prime:
\[P=\langle a,b\mid a^{p^2}=b^p=1,[a,b]=a^p\rangle.\]
Any two elements of $\langle a\rangle$ generate a cyclic group; and the
group generated by $b$ and $x$ is cyclic if $x=1$, abelian but not cyclic
if $x=a^p$, and non-abelian if $x=a$.

The proof is by induction on the number $n$ of vertices. The result is clearly
true if $n=1$. So let $\{v_1,\ldots,v_n\}$ be the vertex set of $\Gamma$,
and suppose that we have an embedding of $\{v_1,\ldots,v_{n-1}\}$ into a
group $G$ satisfying (a)--(c).

Choose an odd prime $p$ not dividing $|G|$, and consider the group $P\times G$,
where $P$ is as above. Modify the embedding of the first $n-1$ vertices by
replacing $v_i$ by $(1,v_i)$ if $\{v_i,v_n\}$ is red, by $(a^p,v_i)$ if
$\{v_i,v_n\}$ is green, and by $(a,v_i)$ if $\{v_i,v_n\}$ is blue. It is
easily checked that we still have an embedding of $\{v_1,\ldots,v_{n-1}\}$
satisfying (a)--(c). Moreover, if we now embed $v_n$ as $(b,1)$, we find that
the conditions hold for the remaining pairs as well. \qed
\end{proof}

Clearly there are plenty of problems along similar lines to investigate here.

\section{Products}

There are a number of graph products. Here I will be chiefly concerned with
the strong product, defined as follows.

Let $\Gamma$ and $\Delta$ be graphs with vertex sets $V$ and $W$ respectively.
The \emph{strong product} $\Gamma\boxtimes\Delta$ has vertex set the Cartesian
product $V\times W$; vertices $(v_1,w_1)$ and $(v_2,w_2)$ are joined whenever
$v_1$ is equal or adjacent to $v_2$ and $w_1$ is equal or adjacent to $w_2$,
but not equality in both places. (All of the graphs in the hierarchy naturally
have loops at each vertex, which we have discarded; the strong product is the
natural categorical product in the category of graphs with a loop at each 
vertex.)

I note that the strong product, along with the Cartesian and categorical
products, is denoted by a symbol representing the corresponding product of
two edges: the Cartesian product is $\Gamma\mathbin{\Box}\Delta$, while the
categorical product is $\Gamma\times\Delta$.

The question of the perfectness of strong products of graphs has been
studied by Ravindra and Parthasarathy~\cite{rp}.

The only group product that concerns us here is the direct product.

\begin{prop}
Let $G$ and $H$ be finite groups.
\begin{enumerate}
\item $\Com(G\times H)=\Com(G)\boxtimes\Com(H)$.
\item If $G$ and $H$ have coprime orders, then
$\EPow(G\times H)=\EPow(G)\\boxtimes\EPow(H)$.
\item If $G/G'$ and $H/H'$ have coprime orders, and in particular if $G$ and
$H$ are perfect groups, then $\DCom(G\times H)=\DCom(G)\boxtimes\DCom(H)$.
\end{enumerate}
\end{prop}

\begin{proof}
(a) Distinct elements $(g_1,h_1)$ and $(g_2,h_2)$ in $G\times H$ commute if
and only if $g_1$ and $g_2$ are equal or commute, and $h_1,h_2$ are equal or
commute.

\medskip

(b) Suppose that $|G|$ and $H$ are coprime. If $\langle g_1,g_2\rangle$ and
$\langle h_1,h_2\rangle$ are cyclic, then (as their orders are coprime) their
direct product is also cyclic and contains $(g_1,h_1)$ and $(g_2,h_2)$.
Conversely, again using coprimeness, if $\langle(g_1,h_1),(g_2,h_2)\rangle$
is cyclic, then it contains $(g_1,1)$, $(g_2,1)$, $(1,h_1)$ and $(1,h_2)$.

\medskip

(c) A formula of Schur gives the Schur multiplier of $G$ and $H$ to be
$M(G)\times M(H)\times(G\otimes H)$, where $M(G)$ is the Schur multiplier of
$G$. If $|G/G'|$ and $|H/H'|$ are coprime, the third term is absent. It
follows that a Schur cover of $G\times H$ is the direct product of Schur
covers of $G$ and $H$. The result follows. \qed
\end{proof}

Thus, questions about the commuting graph or enhanced power graph of a 
nilpotent group can be reduced to questions about the corresponding graphs
for their Sylow subgroups.

The corresponding result fails for the power graph and the non-generating
graph. The power graphs of $C_2$ and $C_3$ are complete but the power graph
of $C_2\times C_3$ are not. For the non-generating graph, we note that for
any non-abelian finite simple group $G$, there is an integer $m$ such that
$G^n$ fails to be $2$-generated if $n>m$.

\section{Cographs and twin reduction}

Two vertices in a graph are called twins if they have the same neighbours
(possibly excluding one another). Equivalently, $v$ and $w$ are twins if
the transposition $(v,w)$ (fixing the other vertices) is an automorphism
of $\Gamma$. If $G$ is a non-trivial group, then any of the graphs in our
hierarchy based on $G$ will contain many pairs of twins. Thus, twins will play
an important part when we come to look at automorphism groups.

If a graph has twins, then we can make a new graph by merging the twins to a
single vertex. The process can be continued until no pairs of twins remain.
If the resulting graph has just a single vertex, the original graph is called
a cograph.

Cographs also play an important part in the story, and make another link with
the Gruenberg--Kegel graph. So we make a detour to look at twin reduction 
and cographs.

\subsection{Cographs}

A graph $\Gamma$ is a \emph{cograph} if either of the following equivalent
conditions holds for it:
\begin{itemize}\itemsep0pt
\item $\Gamma$ does not contain the four-vertex path $P_4$ as an induced
subgraph;
\item $\Gamma$ can be constructed from the $1$-vertex graph by the operations
of complement and disjoint union.
\end{itemize}
In particular, a cograph is connected if and only if its complement is
disconnected. This leads to a tree representation of cographs and to very
efficient algorithms for determining their properties.

I remark here that a connected component of a cograph has diameter at
most~$2$, since two vertices at distance~$3$ would be joined by an induced
$4$-vertex path.

Here is a curious fact which may (or may not) have a connection with
the following material. The \emph{$P_4$-structure} of a graph $\Gamma$ is
the hypergraph whose vertex set is the same as that of $\Gamma$, the 
hyperedges being the $4$-element sets which induce a copy of $P_4$ in
$\Gamma$. We have the following easy observations:
\begin{itemize}\itemsep0pt
\item a graph and its complement have the same $P_4$-structure, since the
graph $P_4$ is self-complementary;
\item a graph is a cograph if and only if its $P_4$-structure is the null
hypergraph.
\end{itemize}
Bruce Reed~\cite{reed} proved the \emph{semi-strong perfect graph
theorem}, which had been conjectured by Vasek Chv\'atal, asserting that if two
graphs have isomorphic $P_4$-structures and one is perfect, then so is the
other.

Cographs have been rediscovered a number of times, and as a result appear in
the literature with very different names, such as ``complement-reducible
graphs'', ``hereditary Dacey graphs'' and ``N-free graphs''. See
\cite{seinsche,sumner,jung} for information about cographs.

\subsection{Twins and twin reduction}

In a graph $\Gamma$, we can define two kinds of ``twin relations'' on vertices.
The \emph{open neighbourhood} $\Gamma(v)$ of $v$ in $\Gamma$ is the set of
vertices in $\Gamma$ joined to $v$; the \emph{closed neighbourhood} is
$\Gamma(v)\cup\{v\}$.
Two vertices $v,w$ are \emph{open twins} if they have the same open
neighbourhoods; they are \emph{closed twins} if they have the same closed
neighbourhoods. Both of these relations are obviously equivalence relations.
Two vertices are open twins in $\Gamma$ if and only if they are closed twins
in the complement of $\Gamma$. Note that open twins are not joined, while
closed twins are joined.

If $x$ and $y$ are twins, then we may collapse them to a single vertex; this
process is called \emph{twin reduction}. Some, though clearly not all,
properties of a graph are preserved by twin reduction. One such property is
that of being a cograph, since twin reduction can neither create nor destroy
an induced $4$-vertex path (since no two of its vertices can be twins). We
will see a more general result shortly.

\subsection{The cokernel of a graph}

\begin{theorem}
Given a graph $\Gamma$, the result of performing a sequence of twin reductions
until the graph is twin-free is unique up to isomorphism, independent of the
chosen sequence of reductions.
\end{theorem}

\begin{proof}
Open and closed twin classes of sizes greater than $1$ are disjoint. For
suppose that $x$ and $y$ are open twins and $y$ and $z$ are closed twins. Then
$xy$ is a non-edge while $yz$ is an edge. Since $y$ and $z$ are twins, $x$
is not joined to $z$; but since $y$ and $x$ are twins, $z$ is joined to $x$.
These conclusions are contradictory. (Alternatively, since $y$ and $z$ are
twins, $(y,z)$ is an automorphism, and so $x$ and $z$ are open twins; and
similarly using the automorphism $(x,y)$, $x$ and $z$ are closed twins.)

We are going to prove the theorem by induction on the number of vertices. There
is nothing to do for graphs with a single vertex, so let $\Gamma$ have $n$
vertices, with $n>1$, and assume that the result is true for any graph with
fewer than $n$ vertices. Take two twin reduction sequences on $\Gamma$. Suppose
that the first begins by identifying $x$ and $y$, and the second by identifying
$u$ and $v$. 

If $\{x,y\}$ = $\{u,v\}$, then the two sequences result in the same graph,
and induction finishes the job.

If $|\{x,y\}\cap\{u,v\}|=1$, then our initial remark shows that the two pairs
of twins have the same type, so the graphs obtained after one step are
isomorphic, and again induction finishes the job.

Suppose that $\{x,y\}\cap\{u,v\}=\emptyset$. Then the two reductions commute.
Let $\Delta$ be the graph obtained by applying the two reductions. Then
$\Delta$ occurs after two steps in reduction sequences for $\Gamma$ beginning
by identifying $x$ and $y$, or by identifying $u$ and $v$. By induction the
end result of either given sequence is the same as the result of reducing
$\Delta$ (up to isomorphism).

The theorem is proved. \qed
\end{proof}

The next result gives the connection between cographs and twin reduction.

\begin{prop}
A graph $\Gamma$ is a cograph if and only if the cokernel of $\Gamma$ is the
graph with a single vertex.
\end{prop}

\begin{proof}
For the necessity, we show by induction that a cograph with more than
one vertex contains twins. Let
$\Gamma$ be a cograph with more than one vertex, and suppose that any smaller
cograph with more than one vertex contains twins.
If $\Gamma$ is disconnected, then if it has a component with more than one
vertex, then this component contains twins; otherwise $\Gamma$ is a null graph
and all pairs of vertices are open twins. If $\Gamma$ is connected, then its
complement is disconnected, and we argue in the complement instead.

Now the result of twin reduction is an induced subgraph of $\Gamma$, and so
also a cograph; so so the reduction continues until only one vertex remains.

Conversely, suppose that $\Gamma$ is not a cograph. Then $\Gamma$ contains
a $4$-vertex path, say $(w,x,y,z)$. Then any pair of these vertices are not
twins, and so are not identified in any twin reduction; so the result of the
reduction still contains a $4$-vertex path. So no sequence of reductions can
terminate in a single vertex. \qed
\end{proof}

This gives another test for a cograph: apply twin reductions until the process
terminates, and see whether just one vertex remains.

\subsection{Cographs as comparability graphs}

A partially ordered set is \emph{N-free} if it does not contain as an induced
substructure the poset whose Hasse diagram is shown in Figure~\ref{f:n}.
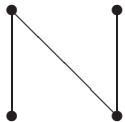
\begin{figure}[htbp]
\begin{center}
\setlength{\unitlength}{0.7mm}
\begin{picture}(20,20)
\multiput(0,0)(20,0){2}{\circle*{2}}
\multiput(0,20)(20,0){2}{\circle*{2}}
\multiput(0,0)(20,0){2}{\line(0,1){20}}
\put(20,0){\line(-1,1){20}}
\end{picture}
\caption{\label{f:n}The poset N}
\end{center}
\end{figure}

N-free posets form the smallest class of posets containing the $1$-element
poset and closed under disjoint union and ordered sum. (In the disjoint
union of two posets, elements of different posets are incompatible; in the
ordered sum, every element of the first is smaller than every element of
the second.)

N-free posets arise in statistical design, where the operations of
\emph{crossing} and \emph{nesting} correspond to disjoint union and ordered
sum. An orthogonal block structure is \emph{special} if it can be built by
crossing and nesting. These constructions also apply to association schemes.
Crossing and nesting were introduced by Nelder~\cite{nelder} and the more
general poset constructions by Speed and Bailey~\cite{sb}; see \cite{bailey}
for more detail and historical comments.

\begin{prop}
\begin{enumerate}
\item The comparability graph of an N-free poset is a cograph.
\item Conversely, any cograph is the comparability graph of an N-free
poset.
\end{enumerate}
\end{prop}

\begin{proof}
The first statement is immediate, and the second is easily proved by
induction. \qed
\end{proof}

\subsection{Twin reduction in the hierarchy}

The relevance of twin reduction to our problem is:

\begin{prop}\label{p:twins}
Let $\Gamma$ be the power graph, enhanced power graph, deep commuting graph,
commuting graph, or non-generating graph of a non-trivial group $G$.
Then the twin relation on $\Gamma$ is not the relation of equality.
\end{prop}

\begin{proof}
Suppose that $G$ contains an element $g$ of order greater than~$2$. Let $h$
be an element such that $g\ne h$ but $\langle g\rangle=\langle h\rangle$. Then
any element joined to one of $g$ and $h$ in one of the graphs listed is also
joined to the other. The arguments are all easy; let us look at the least
trivial, the deep commuting graph. Let $H$ be a Schur cover of $G$ with kernel
$Z$, and $x$ and $y$ elements of $H$ covering $g$ and $h$ respectively. Then
$\langle Z,x\rangle$ is abelian and contains $y$, so $x$ and $y$ commute.

The groups not covered by this are elementary abelian $2$-groups. In these
cases, everything is clear: the power graph, enhanced power graph and deep
commuting graph are stars; the commuting graph is complete; the non-generating
graph is complete if the group has order greater than~$4$. All these graphs
are cographs. \qed
\end{proof}

So, for any of our classes of graphs, say $\mathrm{X}$,  the question ``What
is the cokernel of $\mathrm{X}(G)$?'' is a generalisation of ``Is
$\mathrm{X}(G)$ a cograph?''

Finally for this section, I note that the class of cographs is not preserved
by strong product. Figure~\ref{f:p4p3} shows $P_4$ as an induced subgraph of
$P_3\boxtimes P_3$.

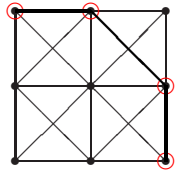
\begin{figure}[htb]
\begin{center}
\setlength{\unitlength}{1mm}
\begin{picture}(20,20)
\multiput(0,0)(10,0){3}{\circle*{1}}
\multiput(0,10)(10,0){3}{\circle*{1}}
\multiput(0,20)(10,0){3}{\circle*{1}}
\multiput(0,0)(10,0){3}{\line(0,1){20}}
\multiput(0,0)(0,10){3}{\line(1,0){20}}
\put(0,0){\line(1,1){20}}
\multiput(0,10)(10,-10){2}{\line(1,1){10}}
\put(0,20){\line(1,-1){20}}
\multiput(0,10)(10,10){2}{\line(1,-1){10}}
\color{red}
\multiput(0,20)(10,0){2}{\circle{2}}
\multiput(20,0)(0,10){2}{\circle{2}}
\color{black}\thicklines
\put(0,20){\line(1,0){10}}
\put(20,0){\line(0,1){10}}
\put(10,20){\line(1,-1){10}}
\end{picture}
\end{center}
\caption{\label{f:p4p3}Cographs are not closed under $\boxtimes$}
\end{figure}

\subsection{Forbidden subgraphs and twin reduction}

For a graph $F$, let $\Forb{F}$ denote the class of graphs containing no
induced subgraph isomorphic to $F$.

If $F$ is twin-free, then it is clear that identifying a pair of twins in
a graph $\Gamma$ can neither create nor destroy an induced copy of $F$. So
the class $\Forb(F)$ and its complement are both preserved by twin reduction,
and in particular:

\begin{theorem}
If $F$ is twin-free, then a graph belongs to $\Forb{F}$ if and only if its
cokernel belongs to $\Forb{F}$.
\end{theorem}

The theorem immediately extends from a single graph $F$ to an arbitrary class
$\mathcal{F}$ of graphs.

The Strong Perfect Graph Theorem shows that perfect graphs form the class
$\Forb(\mathcal{F})$, where $\mathcal{F}$ consists of all cycles of odd
length greater than $3$ and their complements. All these graphs are twin-free.
So as a corollary of the theorem, we see:

\begin{corollary}
A graph is perfect if and only if its cokernel is perfect.
\end{corollary}

This has the immediate consequence that cographs are perfect, although this
is easily proved directly.

However, there are interesting graph classes defined by forbidden subgraphs
which are not twin-free, including chordal graphs, split graphs, and
threshold graphs. We now turn to the general question.

\section{Forbidden subgraphs}

In this section we consider various classes of graphs defined by forbidden
induced subgraphs, and ask: for which finite groups $G$ can one of the graphs
in the hierarchy defined on $G$ belong to this class? Not much is known.

We saw that a graph is perfect if and only if it has no induced subgraph
isomorphic to an odd cycle of length greater than $3$ or the complement of one.
In particular, a cograph is perfect.

We have seen that power graphs are comparability graphs of partial orders,
and noted that these graphs are perfect. For our other types of graph, no
proper subclass defined by forbidden induced subgraphs contains them all; so
we have to ask a  different question: which groups have the property that one
of these graphs belongs to a graph class defined in this way?

\subsection{Power graphs}

We begin with the question: When is the power graph of the group $G$ a
cograph? This question was considered in the paper \cite{cmm}; here the
Gruenberg--Kegel graph makes another appearance. The question is answered
completely for nilpotent groups, but a necessary and a sufficient condition
are known for general groups in terms of the GK graph; these conditions do
not coincide, and we will see that no condition just in terms of the GK
graph can be both necessary and sufficient.

I begin with a simple observation.

\begin{prop}
Let $G$ be a group of prime power order. Then $\Pow(G)$ is a cograph.
\end{prop}

\begin{proof}
We use the fact that, if $G$ is cyclic of prime power order, then $\Pow(G)$
is the complete graph. Suppose that $(a,b,c)$ is a $3$-vertex induced path. In
$\DPow(G)$, we cannot have $a\to b\to c$, since transitivity would imply
$a\to c$, contrary to assumption; similarly, $c\to b\to a$ is impossible.
If $b\to a$ and $b\to c$, then $a,c\in\langle b\rangle$, a cyclic $p$-group,
and so $a\sim b$ in the power graph, also a contradiction. So the only 
possibility is $a\to b$ and $c\to b$.

Now suppose that $(a,b,c,d)$ is a $4$-vertex path. Then considering the 
subpath $(b,c,d)$ shows that $b\to c$ and $d\to c$. But $a\to b\to c$ is
impossible; so no such path can exist. \qed
\end{proof}

\begin{theorem}
Let $G$ be a finite nilpotent group. Then $\Pow(G)$ is a cograph if and only
if either
\begin{enumerate}\itemsep0pt
\item $G$ has prime power order; or
\item $G$ is cyclic of order $pq$, where $p$ and $q$ are distinct primes.
\end{enumerate}
\end{theorem}

\begin{theorem}\label{t:power_nilp}
\begin{enumerate}\itemsep0pt
\item Let $G$ be a finite group whose Gruenberg--Kegel graph is a null graph.
Then the power graph of $G$ is a cograph.
\item Let $G$ be a finite group whose power graph is a cograph. Then, with
possibly one exception, a connected component of the Gruenberg--Kegel graph
has at most two vertices, the exception being the component containing the
prime~$2$. If $\{p,q\}$ is a connected component of the GK graph, with $p$
and $q$ odd primes, then $p$ and $q$ divide $|G|$ to the first power only.
\end{enumerate}
\end{theorem}

For the final claim, note that the groups $\mathrm{PSL}(2,11)$ and $M_{11}$
have the same GK-graph (an edge $\{2,3\}$ and isolated vertices $\{5\}$ and
$\{11\}$), but the power graph of the first is a cograph, that of the second
is not.

\begin{question}
Classify the groups whose power graph is a cograph.
\end{question}

The next question is maybe not so important in itself but might be an indication
of how well we understand groups whose power graph is a cograph.

\begin{question}
Is the following true? Let $\Gamma$ be a cograph. Then there is a finite group
$G$ such that $\Pow(G)$ is a cograph and contains $\Gamma$ as an induced
subgraph.
\end{question}

There are two reasons why this might be tricky:
\begin{itemize}
\item The inductive scheme for cographs requires complementation. But this
does not work nicely for power graphs.
\item Theorem~\ref{t:power_nilp} shows that there will be no simple
construction using direct products of groups of coprime order.
\end{itemize}

Probably this question will involve groups of prime power order.

\medskip

The remainder of this subsection is based on the paper~\cite{cmm}.

We saw that the power graph of a finite group is a comparability graph of a
partial order, and so in particular is a perfect graph. Various interesting
subclasses of the perfect graphs are defined by forbidding certain induced
subgraphs. Cographs form an example: they forbid the $4$-vertex path $P_4$.
Here are some other graph classes.

\begin{itemize}\itemsep0pt
\item A graph is \emph{chordal} if it contains no induced cycles of length
greater than~$3$ (that is, every cycle of length greater than~$3$ has a
chord).
\item A graph is \emph{split} if the vertex set is the disjoint union of two
subsets, one inducing a complete graph and the other a null graph (with 
possibly some edges between them). A graph is split if it contains no induced
subgraph isomorphic to $C_4$, $C_5$ or $2K_2$.
\item A graph is \emph{threshold} if it can be constructed from the $1$-vertex
graph by adding vertices joined either to all or to no existing vertices. A
graph is threshold if and only if it contains no induced subgraph isomorphic
to $P_4$, $C_4$, or $2K_2$.
\end{itemize}

\begin{theorem}
For a finite nilpotent group $G$, the power graph of $G$ is chordal if and only
if one of the following conditions holds:
\begin{enumerate}\itemsep0pt
\item $G$ has prime power order;
\item $G$ is the direct product of a cyclic group of $p$-power order and a
group of exponent $q$, where $p$ and $q$ are distinct primes.
\end{enumerate}
\end{theorem}

\begin{question}
Which non-nilpotent groups have the property that the power graph is chordal?
\end{question}

\begin{theorem}
The following conditions for a finite group are equivalent:
\begin{enumerate}\itemsep0pt
\item $\Pow(G)$ is a threshold graph;
\item $\Pow(G)$ is a split graph;
\item $\Pow(G)$ contains no induced subgraph isomorphic to $2K_2$;
\item $G$ is cyclic of prime power order, or an elementary abelian or dihedral
$2$-group, or cyclic or order $2p$, or dihedral of order $2p^n$ or $4p$, where
$p$ is an odd prime.
\end{enumerate}
\end{theorem}

Note that this theorem does not assume that $G$ is nilpotent.

\subsection{Other graphs}

For our other classes of graphs, the problem of deciding for which groups
the graph in question forbids a certain induced subgraph has not been much
worked on, and a number of interesting questions are open. One difference is
that, as we have seen, every finite graph occurs as an induced subgraph of
each type of graph on some finite group.

In line with the preceding subsection we could ask a multipart question:

\begin{question}\label{q:forb}
For which finite groups is the enhanced power graph/deep commuting
graph/commuting graph/nongenerating graph a perfect graph, or a cograph,
or a chordal graph, or a split graph, or a threshold graph?
\end{question}

We can give a partial answer for groups of prime power order.

\begin{theorem}
Let $G$ be a group of prime power order.
\begin{enumerate}\itemsep0pt
\item The power graph of $G$ is equal to the enhanced power graph, and
contains no induced $P_4$ or $C_4$.
\item If $G$ is $2$-generated of order $p^n$, then the non-generating graph of
$G$ consists of $p+1$ complete subgraphs of order $p^{n-1}$, any two 
intersecting in the same subset of size $p^{n-2}$.
\end{enumerate}
\end{theorem}

\begin{proof}
(a) In a group of prime power order, if two elements generate a cyclic
subgroup, then one is a power of the other; this shows that the power graph
and enhanced power graph coincide. Suppose that $(x,y,z)$ is an induced path of
length~$2$. If $x,z\in\langle y\rangle$, or if $y\in\langle x\rangle$ and
$z\in\langle y\rangle$, or if $x\in\langle y\rangle$, $y\in\langle z\rangle$,
then $x$ and $z$ are joined, a contradiction. So
$y\in\langle x\rangle\cap\langle z\rangle$. Now if $w$ is joined to $x$ but
not to $y$, then $(w,x,y)$ is an induced path of length $2$, so
$x\in\langle y\rangle$ by the same argument. But then $y\in\langle z\rangle$,
so $x\in\langle z\rangle$, a contradiction.

\medskip

(b) Let $\Phi(G)$ be the Frattini subgroup of $G$. By the Burnside Basis
Theorem, $\Phi(G)$ consists of vertices lying in no generating pair and has
index $p^2$ in $G$; moreover, the generating pairs are all pairs of elements
lying in distinct non-trivial cosets of $\Phi(G)$. \qed
\end{proof}

No such result holds for the commuting graph; Theorem~\ref{t:ind_com} shows
that the commuting graphs of $2$-groups form a universal class. 
Computation shows that the smallest group whose commuting graph is not a
cograph is the symmetric group $S_4$: the elements $(1,2,3,4)$, $(1,3)(2,4)$,
$(1,2)(3,4)$ and $(1,3,2,4)$ induce a $4$-vertex path. In fact, seven groups
of order~$32$ have commuting graphs which are not cographs.

\medskip

Question~\ref{q:forb} has been considered for the commuting graph by Britnell
and Gill~\cite{bg}, who obtained a partial description of groups for which the
commuting graph is a perfect graph. Assuming that $G$ has a component (a
subnormal quasisimple subgroup), they determine all possible components of
such groups.

Also, the paper \cite{zbm} classifies nilpotent groups whose enhanced power
graphs are perfect (but I have not been able to access this paper).

\begin{question}
What about other graph classes, for example planar graphs?
\end{question}

For planarity, this may not be too hard for graphs in the hierarchy. Since the
complete graph on $5$ vertices is not planar, it follows that $G$ has no
elements of order greater than $4$. (For all cases except the power graph, a
cyclic subgroup induces a complete graph, so the claim is clear. In the power
graph, the power graph of a cyclic group of prime power order is complete, and
the power graph of $C_6$ contains a $K_5$.) So $G$ is solvable.

Indeed, bounding the genus of a graph bounds its clique number, and so (for
graphs in the hierarchy) bounds the orders of elements.

See~\cite{afk} for the commuting graph and its complement.

\subsection{Results for simple groups}

\begin{prop}
Let $G=\mathrm{PSL}(2,q)$, with $q$ a prime power and $q\ge4$.
\begin{enumerate}\itemsep0pt
\item If $q$ is even, then $\EPow(G)$, $\DCom(G)$ and $\Com(G)$ are cographs;
$\Pow(G)$ is a cograph if and only if $q-1$ and $q+1$ are either prime
powers or products of two distinct primes.
\item If $q$ is odd, then $\EPow(G)$ and $\DCom(G)$ are cographs;
$\Pow(G)$ is a cograph if and only if $(q-1)/2$ and $(q+1)/2$ are either
prime powers or products of two distinct primes.
\end{enumerate}
\end{prop}

\begin{proof}
We consider the graphs obtained by removing the identity. We note that a
graph $\Gamma$ is a cograph if and only if the graph obtained by adding a
vertex joined to all others is a cograph.

We begin with the case when $q$ is a power of $2$, noting that in this case
$\mathrm{PSL}(2,q)$ has trivial Schur multiplier except when $q=4$. The
elements of the group have order $2$ or prime divisors of $q-1$ or $q+1$. The
centralisers of involutions are elementary abelian of order $q$, while the
centralisers of other elements are cyclic of order $q-1$ and $q+1$. In other
words, centralisers are abelian and intersect only in the identity. So the
commuting graph (which is the deep commuting graph if $q\ne4$) is a disjoint
union of complete graphs; the enhanced power graph is a disjoint union of
complete graphs and isolated vertices (corresponding to the involutions);
and the power graph is a disjoint union of the power graphs of cyclic groups
of orders $q\pm1$ and isolated vertices. Using the fact that the power graph
of the cyclic group $C_m$ is a cograph if and only if $m$ is either a prime
power or the product of two primes \cite[Theorem 3.2]{cmm}, the result
follows. The result for the deep commuting graph for $\mathrm{PSL}(2,4)$ can
be proved by computation, or by identifying this group with
$\mathrm{PSL}(2,5)$ (which is dealt with in the next paragraph).

Now consider the case when $q$ is a power of an odd prime $p$. The
centralisers of elements of order $p$ are elementary abelian of order $q$;
centralisers of other elements are cyclic of orders $(q\pm1)/2$ except for
involutions, which are centralised by dihedral groups of order $q\pm1$,
whichever is divisible by $4$. So, although centralisers may not be disjoint,
the maximal cyclic subgroups are, so the statements about the enhanced power
graph and the commuting graph are true.

If $q$ is odd and $q\ne9$, the Schur multiplier of $\mathrm{PSL}(2,q)$ is
cyclic of order $2$, and so a Schur cover of this group is $\mathrm{SL}(2,q)$.
The unique involution in this group is $-I$, the kernel of the extension; so
involutions in $\mathrm{PSL}(2,q)$ lift to elements of order $4$, which lie
in unique maximal abelian subgroups.

The group $\mathrm{PSL}(2,9)$ has Schur multiplier of order~$6$. The deep
commuting graph of $\mathrm{PSL}(2,9)$ was analysed using \textsf{GAP}, with
generators for the Schur cover from the on-line Atlas of Finite Group
Representations~\cite{atlas}. \qed
\end{proof}

\begin{question}
Are there infinitely many prime powers $q$ for which the power graph of
$\mathrm{PSL}(2,q)$ is a cograph?
\end{question}

Here is a preliminary analysis of this question.

\paragraph{Case $q$ even} Then $q=2^d$, say, and each of $q+1$ and $q-1$ is
either a prime power or the product of two primes. By Catalan's conjecture (now
Mih\u{a}ilescu's Theorem: see \cite[Section 6.11]{cohn}), the only two proper
powers differing by $1$ are $8$ and $9$. So, unless $q=8$, we conclude that
each of $q+1$ and $q-1$ is either prime or the product of two primes. Moreover,
one of these numbers is divisible by $3$, so we can say further that one of
$q+1$ and $q-1$ is three times a prime, while the other is a prime or the
product of two primes. For example, $2^{11}-1=23\cdot89$ while
$2^{11}+1=3\cdot683$. The values of $d$ up to $200$ for which $q=2^d$ satisfies
the condition are $1$, $2$, $3$, $4$, $5$, $7$, $11$, $13$, $17$, $19$, $23$,
$31$, $61$, $101$, $127$, $167$, $199$.

\paragraph{Case $q$ odd} If $q$ is congruent to $\pm1$ (mod~$8$), then one
of $(q+1)/2$ and $(q-1)/2$ is divisible by $4$, and so must be a power of $2$;
so either $q=9$, or $q$ is a Fermat or Mersenne prime. So, apart from this
case, $q$ is congruent to $\pm3$ (mod~$8$). Now either $q$ is an odd power 
of $3$, or one of $(q+1)/2$ and $(q-1)/2$ is twice a prime, while the other is
three times a prime or a power of $3$ (unless $q=11$ or $q=13$). The odd prime
powers up to $500$ satisfying the condition are $3$, $5$, $7$, $9$, $11$, $13$,
$17$, $19$, $27$, $29$, $31$, $43$, $53$, $67$, $163$, $173$, $243$, $257$,
$283$, $317$.

\medskip

Table~\ref{t:fsg} gives the numbers of vertices in cokernels for small finite
simple groups. Note that a graph is a cograph if and only if its cokernel has
one vertex. The last column of the table gives the number of cyclic subgroups
of $G$, that is, equivalence classes under the relation $\equiv$ where
$x\equiv y$ if $\langle x\rangle=\langle y\rangle$. Equivalent vertices are
closed twins in all our graphs, so this number is an upper bound for the 
number of vertices in each cokernel. (I have replaced $\mathrm{PSL}$ and
$\mathrm{PSU}$ by $L$ and $U$ in the table to save space.)

%\clearpage

\begin{table}[htb]
\[\begin{array}{c|c|cccccc}
G & |G| & \Pow(G) & \EPow(G) & \DCom(G) & \Com(G) & \NGen(G) & \hbox{Cyc} \\
\hline
A_5 & 60 & 1 & 1 & 1 & 1 & 32 & 32 \\
L_2(7) & 168 & 1 & 1 & 1 & 44 & 79 & 79 \\
A_6 & 360 & 1 & 1 & 1 & 92 & 167 & 167 \\
L_2(8) & 504 & 1 & 1 & 1 & 1 & 128 & 156 \\
L_2(11) & 660 & 1 & 1 & 1 & 112 & 244 & 244 \\
L_2(13) & 1092 & 1 & 1 & 1 & 184 & 366 & 366 \\
L_2(17) & 2448 & 1 & 1 & 1 & 308 & 750 & 750 \\
A_7 & 2520 & 352 & 352 & 352 & 352 & 842 & 947 \\
L_2(19) & 3420 & 1 & 1 & 1 & 344 & 914 & 914 \\
L_2(16) & 4080 & 1 & 1 & 1 & 1 & 784 & 784 \\
L_3(3) & 5616 & 756 & 756 & 808 & 808 & 1562 & 1796 \\
U_3(3) & 6048 & 786 & 534 & 499 & 499 & 1346 & 1850 \\
L_2(23) & 6072 & 1267 & 1 & 1 & 508 & 1313 & 1566 \\
L_2(25) & 7800 & 1627 & 1 & 1 & 652 & 1757 & 2082 \\
M_{11} & 7920 & 1212 & 1212 & 1212 & 1212 & 2444 & 2576 \\
\end{array}\]
\caption{\label{t:fsg}Sizes of cokernels of graphs on small simple groups}
\end{table}

The table suggests various conjectures, some of which can be proved. For
example:

\begin{theorem}
Let $G$ be a non-abelian finite simple group. Then $\NGen(G)$ is not a
cograph.
\end{theorem}

\begin{proof}
We consider the reduced graph obtained by deleting the identity vertex.
The reduced non-generating graph of a simple group is connected, and has
diameter at most~$5$ (this follows from the results of Ma, Herzog \emph{et al.},
Shen and Freedman discussed in Section~\ref{s:intersection} below). Also,
the generating graph is connected (see \cite{bgk}), and indeed has diameter~$2$
(see~\cite{bgh}). But the complement of a connected cograph is disconnected.
\qed
\end{proof}

\begin{question}
Find, or estimate, the number of vertices in the cokernel of the non-generating
graph of a finite simple group. In particular, classify simple groups for which
the number of vertices in the cokernel of the non-generating graph is equal to
the number of cyclic subgroups.
\end{question}

The hypothesis of simplicity is essential here. For example, the non-generating
graph of any $2$-generator $p$-group is a cograph (it consists of a star
$K_{1,p+1}$ with the central vertex blown up to a clique of size $p^{n-2}$ 
and the remaining vertices to cliques of size $p^{n-2}(p-1)$, where the group
has order $p^n$). This graph is a cograph: its complement is complete 
multipartite with some isolated vertices.

\section{Connectedness}
\label{s:conn}

In this section, we examine the question of connectedness of these graphs.
All those in the hierarchy (except the null graph) are connected, since the
identity is joined to all other vertices. The question is non-trivial, however,
if we remove vertices joined to all others. The first job is to characterise
these vertices.

\subsection{Centres}

In the commuting graph of $G$, the set of vertices joined to all others is
simply the centre $Z(G)$ of $G$. So we adapt the terminology by defining
analogues of the centre for other graphs in the list. So, if $\mathrm{X}$
denotes $\Pow$, $\EPow$, $\DCom$, $\Com$ or $\NGen$, we define the 
$\mathrm{X}$-centre of $G$, denoted $Z_\mathrm{X}(G)$, to be the set of
vertices joined to all others in $\mathrm{X}(G)$. It turns out that (aside
from the non-generating graph), in almost all cases, $Z_\mathrm{X}(G)$ is a
normal subgroup of $G$; the only exception is for the power graph of a cyclic
group of non-prime-power order. We note that $Z_{\EPow}(G)$ has been studied
by Patrick and Wepsic~\cite{pw}, who called it the
\emph{cyclicizer} of $G$: it is the set
\[\{x\in G:(\forall y\in G)\langle x,y\rangle\hbox{ is cyclic}\}.\]
The main results on the cyclicizer are also given in \cite{ah1}.

\begin{theorem}\label{t:zzzz}
\begin{enumerate}\itemsep0pt
\item $Z_{\Pow}(G)$ is equal to $G$ if $G$ is cyclic of prime power order; or
the set consisting of the identity and the generators if $G$ is cyclic of 
non-prime-power order; or
$Z(G)$ if $G$ is a generalized quaternion group; or $\{1\}$ otherwise.
\item $Z_{\EPow}(G)$ is the product of the Sylow $p$-subgroups of $Z(G)$ for
$p\in\pi$, where $\pi$ is the set of primes $p$ for which the Sylow
$p$-subgroup of $G$ is cyclic or generalized quaternion; in particular,
$Z_{\EPow}(G)$ is cyclic.
\item $Z_{\DCom}(G)$ is the projection into $G$ of $Z(H)$, where $H$ is a Schur
cover of $G$.
\item $Z_{\Com(G)}=Z(G)$.
\end{enumerate}
\end{theorem}

\begin{proof}
(a) This is \cite[Proposition 4]{cameron}.

\medskip

(b) Suppose that $x$ is joined to all other vertices in $\EPow(G)$. Then
$\langle x,y\rangle$ is cyclic for all $y\in G$; so certainly $x\in Z(G)$.

If three elements of a group have the property that any two of them generate
a cyclic group, then all three generate a cyclic group:
see~\cite[Lemma 32]{aetal}. So $Z_{\EPow}(G)$ is a subgroup, since if
$x,y\in Z_{\EPow}(G)$ then, for all $w\in G$, 
$\langle x,w\rangle$ and $\langle y,w\rangle$ are cyclic, and so if 
$\langle x,y\rangle=\langle z\rangle$ then $\langle z,w\rangle$ is cyclic
for all $w\in G$, so that $z\in Z_{\EPow}(G)$.

Now let $x$ be an element of prime order $p$ in $Z_{\EPow}(G)$. If $G$ 
contains a subgroup $C_p\times C_p$ then there is an element of order $p$
not in $\langle x\rangle$, so not adjacent to $x$, a contradiction. So the
Sylow $p$ subgroup of $G$ is cyclic or generalised quaternion, by Burnside's
theorem. But now $x$ lies in every Sylow $p$-subgroup of $G$, so is joined
to every element of $p$-power order, and hence to every element of $G$.

\medskip

(c) and (d) Part (d) is clear, and (c) follows since $\DCom(G)$ is a
projection of the commuting graph of a Schur cover of $G$. \qed
\end{proof}

By contrast, $Z_{\NGen}(G)$ is not necessarily a subgroup of the $2$-generated
group $G$. If $G$ is non-abelian, then $Z_{\NGen}(G)$ must contain $Z(G)$,
since the non-generating graph contains the commuting graph. Also, it contains
the Frattini subgroup $\Phi(G)$ of $G$, since $2$-element generating sets are
minimal and so their elements do not lie in the Frattini subgroup
(which consists of the elements which can be dropped from any generating set). 

If the order of $G$ is a prime power, then the Burnside basis theorem shows
that $Z_{\NGen}(G)=\Phi(G)$, since a set of elements generates $G$ if and only
if its projection onto $G/\Phi(G)$ generates this quotient.

Also, by the result of \cite{bgk} (see Theorem~\ref{t:bgh}), if all proper
quotients of $G$ are cyclic, then $Z_{\NGen}(G)=\{1\}$.

In general, however, $Z_{\NGen}(G)$ may be a subgroup different from both
$Z(G)$ and $\Phi(G)$. For example, let $G$ be the symmetric group $S_4$. Then
both $Z(G)$ and $\Phi(G)$ are trivial, but $Z_{\NGen}(G)$ is the Klein group
$V_4$ (the minimal normal subgroup of $G$).

Moreover, it may not be a subgroup at all. For example, if $G=C_6\times C_6$,
then $Z_{\NGen}(G)$ consists of the elements not of order~$6$, since both
elements in any generating pair must have order $6$.

\begin{question}
Characterise the $2$-generated groups in which $Z_{\NGen}(G)$ is a subgroup
of $G$.
\end{question}

\subsection{Connectedness}

Each of our types of graph is connected, since the corresponding ``centre''
is non-empty and its vertices are joined to all others. So the question
becomes interesting if we ask whether the induced subgraph on the elements
outside this centre is connected.

The situation for the commuting graph is well-understood, thanks to the
results of \cite{gp,mp}. But first I mention another link with the
Gruenberg--Kegel graph. This has been known for some time, but the first
mention I know in the literature is \cite[Section~3]{mp}.

\begin{theorem}\label{t:connectedcom}
Let $G$ be a group with trivial centre. Then the induced subgraph of the
commuting graph on $G\setminus\{1\}$ is connected if and only if the
Gruenberg--Kegel graph is connected.
\end{theorem}

\begin{proof}
Suppose first that $Z(G)=1$ and the commuting graph is connected. Let $p$ and
$q$ be primes dividing $|G|$. Choose elements $g$ and $h$ of orders $p$ and $q$
respectively, and suppose their distance in the commuting graph is $d$. We
show by induction on $d$ that there is a path from $p$ to $q$ in the GK
graph.

If $d=1$, then $g$ and $h$ commute, so $gh$ has order $pq$, and $p$ is
joined to $q$.  So assume the result for distances less than $d$, and let
$g=g_0,\ldots,g_d=h$ be a path from $g$ to $h$.

Let $i$ be mimimal such that
$p$ does not divide the order of $g_i$ (so $i>0$). Now some power of $g_{i-1}$,
say $g_{i-1}^a$, has order $p$, while a power $g_i^b$ of $g_i$ has prime
order $r\ne p$.

The distance from $g_i^b$ to $g_d$ is at most $d-i<d$, so
there is a path from $r$ to $q$ in the GK graph. But $g_{i-1}^a$ and
$g_i^b$ commute, so $p$ is joined to $r$.

\medskip

For the converse, assume that the GK graph is connected.

Note first that for every non-identity element $g$, some power of $g$ has prime
order, so it suffices to show that all elements of prime order lie in the same
connected component of the commuting graph. Also, since a non-trivial
$p$-group has non-trivial centre, the non-identity elements of any Sylow
subgroup lie in a single connected component.

Let $C$ be a connected component. Connectedness of the GK graph shows that
$C$ contains a Sylow $p$-subgroup for every prime $p$ dividing $|G|$. Also,
every element of $C$, acting by conjugation, fixes $C$. It follows that the
normaliser of $C$ is $G$, and hence that $C$ contains every Sylow subgroup
of $G$, and thus contains all elements of prime order, as required. \qed
\end{proof}

After a lot of preliminary work, much of it on specific groups, summarised
in the introduction to \cite{gp}, Iranmanesh and
Jafarzadeh~\cite{ij} conjectured that there is an absolute upper bound on the
diameter of any connected component of the induced subgraph on
$G\setminus Z(G)$ of the commuting graph of $G$. This conjecture was refuted
by Giudici and Parker~\cite{gp}. However, it was proved for groups with
trivial centre by Morgan and Parker~\cite{mp}. In these results, I use the
term ``reduced commuting graph'' to mean the induced subgraph of the commuting
graph of $G$ on $G\setminus Z(G)$, in which no vertex is joined to all others.

\begin{theorem}
There is no upper bound for the diameter of the reduced commuting graph of a
finite group; for any given $d$ there is a $2$-group whose reduced commuting
graph is connected with diameter greater than $d$.
\end{theorem}

On the other hand:

\begin{theorem}
Suppose that the finite group $G$ has trivial centre. Then every connected
component of its reduced commuting graph has diameter at most~$10$.
\end{theorem}

For the power graph and enhanced power graph, we note that, if the group
$G$ is not cyclic or generalized quaternion, then the corresponding 
``centre'' is just the identity. So the natural question is: if $G$ is not
cyclic or generalized quaternion, is the induced subgraph of the power
graph on non-identity elements connected? This question has been considered
in several papers, for example \cite{cj,zbm}.

The next result shows that we have only one rather than two problems to
consider.

\begin{prop}
Let $G$ be a group with $Z(G)=\{1\}$. Then the reduced power graph of $G$
is connected if and only if the reduced enhanced power graph of $G$ is
connected.
\end{prop}

\begin{proof}
If $g$ and $h$ are joined in the power graph, they are joined in the enhanced
power graph; if they are joined in the enhanced power graph, then they lie at
distance at most~$2$ in the power graph; both $g$ and $h$ are powers of the
intermediate vertex, which is thus not the identity. \qed
\end{proof}

The argument shows that, if these graphs are connected, the diameter of the
power graph is at least as great as, and at most twice, the diameter of the
enhanced power graph. Can these bounds be improved?

I have already quoted the result of Burness~\emph{et~al.}\ on the generating
graph. For the non-generating graph, the results of Freedman and others on
the difference between the non-generating graph and the commuting graph
(that is, the graph $(\NGen{-}\Com)(G)$) have been mentioned also.

As promised, here is a weak result on the graph $(\Com{-}\Pow)(G)$.

\begin{theorem}
Suppose that the finite group $G$ satisfies the following conditions:
\begin{enumerate}\itemsep0pt
\item The Gruenberg--Kegel graph of $G$ is connected.
\item If $P$ is any Sylow subgroup of $G$, then $Z(P)$ is non-cyclic.
\end{enumerate}
Then the induced subgraph of $(\Com{-}\Pow)(G)$ on $G\setminus\{1\}$ either
has an isolated vertex or is connected.
\end{theorem}

\begin{proof}
Let $\Gamma(G)$ denote the induced subgraph of $(\Com{-}\Pow)(G)$ on
$G\setminus\{1\}$. Note that, if $H$ is a subgroup of $G$, then the induced
subgraph of $\Gamma(G)$ on $H\setminus\{1\}$ is $\Gamma(H)$.

First we show that, if $P$ is a $p$-group, then $\Gamma(P)$ is connected. Let
$Q\leqslant Z(P)$ with $Q\cong C_p\times C_p$. Then the induced subgraph on
$Q\setminus\{1\}$ is complete multipartite with $p+1$ blocks of size $p-1$,
corresponding to the cyclic subgroups of $Q$. So it suffices to show that
any element $z$ of $P\setminus\{1\}$ has a neighbour in $Q\setminus\{1\}$.
We see that $z$ commutes with $Q$ since $Q\leqslant Z(P)$; and
$\langle z\rangle\cap Q$ is cyclic so there is some element of $Q$ not in this
set.

Now let $C$ be a connected component of $\Gamma(G)$ containing an element $z$
of prime order $p$. Since $\Gamma(G)$ is invariant under $\Aut(G)$, in
particular it is normalized by all its elements, so
$\langle C\rangle\leqslant N_G(C)$. In particular, $C$ contains a Sylow
$p$-subgroup of $G$ (one containing the given element of order $p$ in $C$).

If $C$ contains an element of prime order $r$, and $\{r,s\}$ is an edge of the
GK graph, then $G$ contains an element $g$ of order $rs$, then without loss
of generality $g^s\in C$, and $g^s$ is joined to $g^r$ in $\Gamma(G)$, so
also $g^r\in C$. Now connectedness of the GK graph shows that $C$ contains
a Sylow $q$-subgroup of $G$ for every prime divisor of $|G|$. Hence
$|N_G(C)|$ is divisible by every prime power divisor of $|G|$, whence
$N_G(C)=G$.

Finally, let $g$ be any non-identity element of $G$. Choose a maximal cyclic
subgroup $K$ containing $g$. If $C_G(K)=K$, then the generator of $K$
commutes only with its powers, and is isolated in $\Gamma(G)$. If not, then
there is an element of prime order in $C_G(K)\setminus K$. (If
$h\in C_G(K)\setminus K$, then $\langle g,h\rangle$ is abelian but not cyclic,
so contains a subgroup $\langle g\rangle\times C_m$ for some $m$; choose an
element of prime order in the second factor.) This element is joined to
$g$ in the commuting graph but not in the power graph; so $g\in C$. We
conclude that $C=G\setminus\{1\}$, and the proof is done.\qed
\end{proof}

\paragraph{Remarks} The theorem is probably not best possible. Let us consider
the hypotheses.

We saw in Theorem \ref{t:connectedcom} that, for groups with trivial centre
(our main interest here), connectedness of the GK graph is equivalent to
connectedness of the commuting graph, and so is clearly necessary for
connectedness of $(\Com{-}\Pow)(G)$.

The second condition (which is necessary for the above proof) is
probably much too strong. Perhaps it can be weakened to say that the Sylow
subgroups of $G$ are not cyclic or generalized quaternion groups (or,
subgroups of $G$ are not cyclic or generalized quaternion groups (or,
equivalently, that $G$ has no subgroup $C_p\times C_p$ for prime $p$).
Perhaps it is only necessary to assume this for one prime. More work
needed.

\subsection{Connectedness of the complement}

As well as asking whether our graphs, after reduction (removing vertices joined
to all others) are connected, we can ask the same question for the complementary
graphs. For the commuting graph, there is a simple elegant argument, depending
on the following result.

\begin{prop}\label{p:subgp}
Let $\Gamma$ be a graph whose vertex set is a group $G$, and suppose that for
any vertex $g\in G$, the closed neighbourhood of $g$ is a subgroup of $G$.
Then the complementary graph has just one connected component of size larger
than~$1$; this component has diameter at most~$2$.
\end{prop}

\begin{proof}
The isolated vertices in the complement of $\Gamma$ are the vertices whose
closed neighbourhood in $\Gamma$ is the whole of $G$. Let $g_1,g_2$ be two
elements of $G$ which are not isolated in the complement of $\Gamma$. Then
$H_1=\{g_1\}\cup\Gamma(g_1)$ and $H_2=\{g_2\}\cup\Gamma(g_2)$ are subgroups
of $G$. Since a finite group cannot be written as the union of two proper
subgroups (a simple consequence of Lagrange's Theorem), there is a vertex $h$
outside these two subgroups, hence joined to $g_1$ and $g_2$ in the
complement of $\Gamma$. \qed
\end{proof}

\begin{corollary}\label{c:comcon}
If $G$ is a non-abelian finite group, then the complement of the reduced
commuting graph of $G$ is connected with diameter at most~$2$.
\end{corollary}

\begin{proof}
The closed neighbourhood of a vertex $g$ in the commuting graph is its
centraliser. \qed
\end{proof}

What about our other graphs? The deep commuting graph of $G$ is obtained from
the commuting graph of a Schur cover of $G$ by twin reduction; so its
complement, restricted to the non-isolated vertices, is connected with
diameter at most $2$. 

\begin{theorem}
Let $G$ be a finite group which is not a cyclic $p$-group. Then the complement
of the power graph of $G$ has just one connected component, apart from
isolated vertices.
\end{theorem}

\begin{proof}
Let $\Gamma$ be the complement of the power graph of $G$. We begin with a few
remarks.
\begin{enumerate}\itemsep0pt
\item By Theorem~\ref{t:zzzz}, we know set of the isolated vertices in $\Gamma$
(we have excluded the case where every vertex is isolated):
\begin{itemize}\itemsep0pt
\item if $G$ is cyclic but not of prime power order, the identity and the
generators of $G$;
\item if $G$ is generalised quaternion, $Z(G)$;
\item otherwise, just the identity.
\end{itemize}
\item Since every edge of the noncommuting graph is an edge of $\Gamma$, if
$G$ is nonabelian then $G\setminus Z(G)$ is contained in a single component,
by Corollary~\ref{c:comcon}.
So any other component is contained in $Z(G)\setminus\{1\}$.
\item Suppose that $G$ has more than one subgroup of prime order, and let $S$
be the set of elements of prime order. Then the induced subgraph on $S$ is
connected (it is complete multipartite, with the cyclic subgroups as parts).
Moreover, if $g$ is an element whose order is not divisible by some prime
divisor of $|G|$, then $g$ has a neighbour in $S$.
\end{enumerate}

Suppose first that $G$ is not of prime power order. If $Z(G)=1$, then the
result follows from the second remark, so suppose not. Then $Z(G)\cap S$ is
nonempty, so every element whose order fails to be divisible by some prime
divisor of $|G|$ lies in $Z(G)$. In particular, the Sylow subgroups of $G$
all lie in $Z(G)$, so $G$ is abelian. Now we separate into two subcases:
\begin{enumerate}\itemsep0pt
\item Suppose that $G$ is cyclic. If $g\in G$ and $g$ is not a generator, then
there is a prime $p$ such that the power of $p$ dividing the order of $g$ is
smaller than the power $p^m$ which divides $|G|$. If the order of $g$ is a
$p$-power, it is joined to a vertex in $S$. Otherwise, there is an element
of order $p^m$, joined both to $g$ and to an element of $S$.
\item Suppose that $G$ is not cyclic. There is a prime $p$ such that $G$
contains $C_{p^a}\times C_{p^b}$, where $p^a$ is the $p$-part of the
exponent of $G$. For any $h\in G$, $h$ is joined to the generators of one
or other of these two cyclic groups; and each of them is joined to an
element of $S$.
\end{enumerate}

So now suppose that $G$ is a group of prime power order. If $G$ has more than
one subgroup of order $p$, then every element of $G$ has a neighbour in $S$.
Otherwise, by the theorem of Burnside \cite[Theorem 12.5.2]{hall}, $G$ is
either cyclic (which we have excluded) or generalised quaternion (in which
case the vertices of $Z(G)$ are isolated, and the remainder is connected by
the second opening remark). \qed
\end{proof}

I conclude with two questions which should not be difficult.

\begin{question}
What is the best possible upper bound for the diameter of non-trivial 
connected component of the complement of the power graph, and which groups
attain the bound?
\end{question}

\begin{question}
Is it true that the complement of the enhanced power graph has just one
connected component, apart from isolated vertices?
\end{question}

\section{Automorphisms}

Each type of graph in the hierarchy on a group $G$ is preserved by the
automorphism group of $G$. But in almost all cases, the automorphism group
of the graph is much larger. This question has been considered in \cite{agm}.

We saw in Proposition~\ref{p:twins} that any of our hierarchy of graphs
has non-trivial twin relation. So the first thing we need to do is to take
a look at automorphisms of such graphs.

Let $\Gamma$ be a graph. It is clear that its automorphism group $\Aut(\Gamma)$
preserves twin relations on $\Gamma$, and that vertices in a twin equivalence
class can be permuted arbitrarily. It follows by induction that the group
induces an automorphism group on the cokernel $\Gamma^*$ of $\Gamma$, say
$\Aut^-(\Gamma^*)$. We say that the twin reduction on $\Gamma$ is 
\emph{faithful} if $\Aut^-(\Gamma^*)=\Aut(\Gamma^*)$.

Trivially, if $\Gamma$ is a cograph (so that its cokernel is the $1$-vertex
graph), the twin reduction is faithful; we ignore this case.

The reduction process is not always faithful. For a simple example, consider
Figure~\ref{f:tr}.

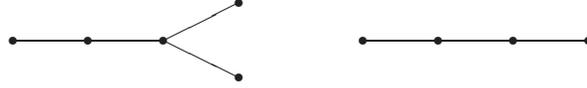
\begin{figure}[htb]
\[\setlength{\unitlength}{1mm}
\begin{picture}(30,10)
\multiput(0,5)(10,0){3}{\circle*{1}}
\multiput(30,0)(0,10){2}{\circle*{1}}
\put(0,5){\line(1,0){20}}
\put(20,5){\line(2,1){10}}
\put(20,5){\line(2,-1){10}}
\end{picture}
\qquad\qquad
\begin{picture}(30,10)
\multiput(0,5)(10,0){4}{\circle*{1}}
\put(0,5){\line(1,0){30}}
\end{picture}\]
\caption{\label{f:tr}Non-faithful twin reduction}
\end{figure}

In the left-hand graph, the two leaves on the right are twins, and twin
reduction gives the right-hand graph as the cokernel. But the cokernel has
an automorphism (reflection in the vertical axis of symmetry) not induced
from an automorphism of the original.

\begin{question}
Given a finite group $G$ and one of our types of graph (say $\mathrm{X}$),
\begin{enumerate}\itemsep0pt
\item When is twin reduction on $\mathrm{X}(G)$ faithful?
\item What is the automorphism group of the cokernel of $\mathrm{X}(G)$?
\end{enumerate}
\end{question}

Very little seems to be known about this question. I first discuss cographs,
then give a couple of examples.

\begin{prop}
Let $\Gamma$ be a cograph. Then the automorphism group of $\Gamma$ can be 
built from the trivial group by the operations of direct product and
wreath product with a symmetric group.
\end{prop}

\begin{proof}
Recall that a cograph can be built from the $1$-vertex graph by the operations
of complement and disjoint union. Now complementation does not change the
automorphism group. If $\Gamma$ is the disjoint union of $m_1$ copies of
$\Delta_1$, \dots, $m_r$ copies of $\Delta_r$, then 
\[\Aut(\Gamma)=\left(\Aut(\Delta_1)\wr S_{m_1}\right)\times\cdots\times
\left(\Aut(\Delta_r)\wr S_{m_r}\right).\]
Assuming inductively that each of $\Aut(\Delta_1)$, \dots, $\Aut(\Delta_r)$
can be built by direct producs and wreath products with symmetric groups;
then the same is true for $\Aut(\Gamma)$. \qed
\end{proof}

\paragraph{Example} Let $G$ be the alternating group $A_5$, and consider the
power graph of $G$. The identity is joined to all other vertices; after removing
it we have six cliques of size $4$ (corresponding to cyclic subgroups of
order~$5$), ten of size~$2$ (corresponding to cyclic subgroups of order~$3$),
and fifteen isolated vertices (corresponding to elements of order~$2$). This
graph is easily seen to be a cograph, so its cokernel has a single vertex. In
fact, closed twin reduction contracts the cliques of sizes $2$ and $4$ to
single vertices, giving a star on $32$ vertices; then open twin reduction
produces a single edge, and closed twin reduction reduces this to a single
vertex.

\paragraph{Example} Let $G$ be the Mathieu group $M_{11}$. The power graph
of $G$ has $7920$ vertices. On removing the identity, we are left with a graph
consisting of 
\begin{itemize}\itemsep0pt
\item $144$ complete graphs of size $10$, corresponding to elements of
order~$11$;
\item $396$ complete graphs of size $4$, corresponding to elements of
order $5$;
\item a single connected component $\Delta$ on the remaining $4895$ vertices.
\end{itemize}
Two steps of twin reduction remove all the components which are complete.
If we take $\Delta$, and first factor out the relation ``same closed
neighbourhood'', and then factor out from the result the relation ``same open
neighbourhood'', we obtain a connected graph on $1210$ vertices whose
automorphism group is $M_{11}$. (This is shown by a \textsf{GAP} computation.)
This group is induced by the automorphism group of the original power graph;
so the reduction is faithful.

\subparagraph{Exercise} Why is the number $1210$ given above two less than
the number of vertices of the cokernel of the power graph of $M_{11}$ given
in Table~\ref{t:fsg}?

\begin{question}
For which non-abelian finite simple group $G$ is it the case that the twin
reduction on the power graph/enhanced power graph/deep commuting
graph/commuting graph/generating graph of $G$ is faithful?
\end{question}

\paragraph{Example}
A curious example showing that this is not true for all such groups
is described in~\cite{clrd}.
Let $G$ be the simple group $\mathrm{PSL}(2,16)$. The automorphism group
of $G$ is the group $\mathrm{P}\Gamma\mathrm{L}(2,16)$, four times as large
as $G$; but the cokernel of the generating graph of $G$ has an extra
automorphism of order~$2$, interchanging the sets of vertices coming from
elements of orders~$3$ and~$5$ in $G$. Twin reduction of this graph is thus
not faithful. The cokernel is a graph on $784$ vertices with automorphism
group $C_2\times\mathrm{P}\Gamma\mathrm{L}(2,16)$.

\medskip

Non-faithfulness means, as in this example, that extra automorphisms are
introduced by twin reduction.

\section{Above and beyond the hierarchy}

This section contains some very brief comments on similar graphs.

\subsection{Nilpontence, solvability and Engel graphs}

Given a subgroup-closed class of graphs $\mathcal{C}$, we can define a graph
on $G$ in which $x$ and $y$ are joined if $\langle x,y\rangle$ belongs to
$\mathcal{C}$.

For $\mathcal{C}$ the class of cyclic groups, we obtain the enhanced power
graph; and for the class of abelian groups, we obtain the commuting graph.

After these, the most natural classes to consider are those of nilpotent and
solvable groups; let us denote the corresponding graphs by $\Nilp(G)$ and
$\Sol(G)$ respectively. 

A \emph{Schmidt group} is a non-nilpotent group all of whose proper subgroups
are nilpotent. These groups were characterised by Schmidt~\cite{schmidt};
see~\cite{bberr} for an accessible account. All are $2$-generated.

I do not know of a similar characterisation of the non-solvable groups all
of whose proper subgroups are solvable. However, we can conclude that they
are $2$-generated, as follows. Let $G$ be such a group, and $S$ the solvable
radical of $G$ (the largest solvable normal subgroup). Then $G/S$ is a
non-abelian simple group. (For if $H/S$ is a minimal normal subgroup of
$G/S$, then $H/S$ is a product of isomorphic simple groups, so $H$ is 
not solvable, and by minimality $H=G$.) Now every finite simple group is
$2$-generated. If we take two cosets $Sg,Sh$ which generate $G/S$, then 
$\langle g,h\rangle$ is a subgroup of $G$ which projects onto $G/S$, and so
is non-solvable; by minimality it is equal to $G$. (In fact we do not need
the Classification of Finite Simple Groups here. For clearly $G/S$ is a
minimal simple group, and so is covered by Thompson's classification of
N-groups~\cite{ngroups}.)

It follows that a group $G$ is nilpotent (resp.\ solvable) if and only if every
$2$-generated subgroup of $G$ is nilpotent (resp.\ solvable). For if every
$2$-generated subgroup of $G$ is nilpotent, then $G$ cannot contain a minimal
non-nilpotent subgroup, and so $G$ is nilpotent; similarly for solvability.

\begin{prop}
\begin{enumerate}
\item For any finite group $G$, we have $E(\Com(G))\subseteq E(\Nilp(G))
\subseteq E(\Sol(G))$.
\item $E(\Com(G))=E(\Nilp(G))$ if and only if all the Sylow subgroups of $G$
are abelian.
\item $E(\Nilp(G))=E(\Sol(G))$ if and only if $G$ is nilpotent.
\item $E(\Com(G))=E(\Sol(G))$ if and only if $G$ is abelian.
\item If $G$ is non-nilpotent, then $E(\Nilp(G))\subseteq E(\NGen(G))$;
equality holds if and only if $G$ is a Schmidt group.
\item If $G$ is non-solvable, then $E(\Sol(G))\subseteq E(\NGen(G))$; equality
holds if and only if $G$ is a minimal non-solvable group.
\end{enumerate}
\end{prop}

\begin{proof}
(a) The first point is clear from the definition.

\medskip

(b) Suppose that $E(\Com(G))=E(\Nilp(G))$. Then two elements from the same
Sylow subgroup of $G$ generate a nilpotent group; hence they commute.
Conversely, if the Sylow subgroups are abelian, then a nilpotent subgroup is
the product of its Sylow subgroups and hence is abelian.

\medskip

(c) Suppose that $E(\Nilp(G))=E(\Sol(G))$. If $G$ is not nilpotent, it 
contains a minimal non-snlpotent subgroup, a Schmidt group, which is
$2$-generated and solvable, hence nilpotent, a contradiction. Conversely, if
$G$ is nilpotent, then $\Nilp(G)$ is complete.

\medskip

(d) If $\Com(G)$ and $\Sol(G)$ coincide, then $G$ is nilpotent with abelian
Sylow subgroups, hence is abelian. The converse is clear.

\medskip

(e), (f) The forward direction in the last two points uses the fact that these
groups are $2$-generated, as remarked above. For if $\Nilp(G)$ and $\NGen(G)$
have the same edges, then two elements which do not generate $G$ must generate
a nilpotent group, and similarly for solvability.\qed
\end{proof}

Regarding (b), groups with all Sylow subgroups abelian are known, since
Walter~\cite{walter} classified the groups with abelian Sylow $2$-subgroups.
The simple groups arising here are $\mathrm{PSL}(2,q)$ with $q$ even or
congruent to $\pm3$~(mod~$8$) and the first Janko group $\mathrm{J}_1$.

The sets of vertices joined to all others in the nilpotency and solvability
graphs have been characterized. The first part of this proposition is due
to Abdollahi and Zarrin, the second part to Guralnick \emph{et al.}

\begin{theorem}
For any finite group $G$,
\begin{enumerate}\itemsep0pt
\item $Z_{\Nilp}(G)$ is the hypercentre of $G$;
\item $Z_{\Sol}(G)$ is the solvable radical of $G$.
\end{enumerate}
\end{theorem}

I will give a proof of the first statement after discssing the Engel graph
below. I refer to the cited paper for the second.
Note that $Z_{\Nilp}(G)$ and $Z_{\Sol}(G)$ are both subgroups of $G$.

This question has been recently studied in greater generality by Lucchini
and Nemmi~\cite{ln}, who investigated the question of when the set of
vertices in the $\mathfrak{F}$-graph of $G$ which are joined to all other
vertices is necessarily a subgroup, iin the case where $\mathfrak{F}$ is a
saturated formation. I will not give details, but refer to their paper.

\begin{question}
Investigate analogues of the earlier results in this paper in the extended
hierarchy of graphs containing $\Nilp(G)$ and $\Sol(G)$.
\end{question}

Universality is relatively straightforward, and both cases can be handled
together. Recall from the proof of Theorem~\ref{t:gen_univ}
that any graph $\Gamma$ can be represented as the intersection graph of a
linear hypergraph, a family of sets with the property that two sets intersect
in $1$ point if the corresponding vertices are adjacent, and are disjoiont
otherwise. We can add some dummy points to ensure that all the sets in the
collection have the same (prime) cardinality $p$ at least~$3$. Now if we take
cycles whose supports are these sets, then we see (as there) that the cycles
corresponding to adjacent vertices generate the alternating group of degree
$2p-1$, while those corresponding to non-adjacent vertices generate
$C_p\times C_p$. So both nilpotence and solvability graphs of finite groups
embed all finite graphs.

Connectedness of the complement of the nilpotency graph has been
investigated by Abdollahi and Zarrin~\cite{az}. There is a generalisation to
$\mathfrak{F}$-groups, for saturated formations $\mathfrak{F}$, by
Lucchini and Nemmi~\cite{ln}.

\medskip

Each of these cases can be stratified: we can define the level-$k$ nilpotence
or solvability graph to have edges $\{x,y\}$ if $\langle x,y\rangle$ is
nilpotent of class at most $k$ (resp.\ solvable of derived length at most
$k$).

\medskip

Other classes of groups for which the corresponding graphs could be studied,
for which the minimal groups not in the class have been considered, include
the supersolvable groups (those for which every chief factor is cyclic)
and the $p$-nilpotent groups (groups with normal $p$-complements). The
groups minimal with respect to not lying in these classes are considered in
\cite{bber} and \cite{bberr} respectively.

\medskip

A closely related graph is the \emph{Engel graph} of a group, defined by
Abdollahi~\cite{a_engel}. Here is a brief account. We define, for each positive
integer $k$, and all $x,y\in G$, the element $[x,{}_ky]$ of $G$ to be the
left-normed commutator of $x$ and $k$ copies of $y$; more formally,
\begin{itemize}\itemsep0pt
\item $[x,{}_1y]=[x,y]=x^{-1}y^{-1}xy$,
\item for $k>1$, $[x,{}_ky]=[[x,{}_{k-1}y],y]$.
\end{itemize}
Abdollahi defined $x$ and $y$ to be adjacent if $[x,{}_k]y\ne1$ and
$[y,{}_kx]\ne1$ for all $k$. To fit with the philosophy of this paper, and
at Abdollahi's suggestion, I will redefine it to be the complement of this
graph. If we do this then we have a similar situation to that arising with
the power graph. If we define the \emph{directed Engel graph} to have
an arc from $x$ to $y$ if $[y,{}_kx]=1$ for some $k$, then the Engel graph
(that is, the complement of the graph as defined in~\cite{a_engel}) is the
graph in which $x$ and $y$ are joined if there is an arc from one to the
other. The directed graph may also have a role to play here.

A similar stratification to that for nilpotence and solvability graphs can
also be defined in this case. In particular, the level~$1$ Engel graph is
just the commuting graph.

Zorn~\cite{zorn} showed that, if a finite group
$G$ satisfies an Engel identity $[x,{}_ky]=1$ for all $x,y$ (for some $k$),
then $G$ is nilpotent; so the finite groups for which the directed Engel graph
is complete are the same as those for which the nilpotency graph is complete.
(For infinite groups, this is not true, though the result has been shown
in a number of special cases.)

So there is a close connection between the Engel graph and the nilpotency
graph. But they are not equal in general. For example, in the group $S_3$,
there is an arc of the directed Engel graph from each element of order~$3$
to each element of order~$2$, but not in the reverse direction.

\begin{question}
What can be said about the relation between the Engel and nilpotency graphs?
In particular, in which groups are they equal?
\end{question}

Armed with this knowledge, we return briefly to the nilpotence graph. First,
another definition. The \emph{upper central series} of a group $G$ is the
sequence of subgroups defined by
\[Z^0(G)=\{1\},\qquad Z^{k+1}(G)/Z^k(G)=Z(G/Z^k(G))\hbox{ for $k\geqslant1$}.\]
The \emph{hypercentre} of $G$ is the union of these subgroups. Thus, if $G$ is
finite, then the hypercentre is $Z^k(G)$, where $k$ is the smallest value
such that $Z^k(G)=Z^{k+1}(G)$.  It is clear that we have $Z^1(G)=Z(G)$, and
subsequent terms are the subgroups of $G$ that project onto the upper central
series of $G/Z(G)$. Hence we can give an alternative definition:
\[Z^{k+1}(G)/Z(G)=Z^k(G/Z(G)).\]

\begin{theorem}
Let $G$ be a finite group. Then the set $Z_{\Nilp}(G)$ of elements of $G$
which are joined to all other elements in the nilpotence graph is equal to
the hypercentre of $G$.
\end{theorem}

\begin{proof}
Suppose that $x$ is joined to all other elements of $G$, so that
$\langle x,y\rangle$ is nilpotent for all $y\in G$. In particular, for all $y$,
there exists $k$ such that $[x,{}_ky]=1$. Baer~\cite{baer} showed that, in a
group with the maximum condition on subgroups (and in particular a finite
group), this implies that $x$ is in the hypercentre of $G$.

For the reverse implication, we have to show that, if $x\in Z^k(G)$, where
$k$ is the least value such that $Z^k(G)=Z^{k+1}(G)$, then for any $y\in G$
we have $\langle x,y\rangle$ nilpotent. We prove this by induction on $k$.
So assume that it is true with $k-1$ replacing $k$ in any group. Then for
$x\in Z^k(G)$ we have $Z(G)x\in Z^{k-1}(G/Z(G))$, so
$Z(G)\langle x,y\rangle/Z(G)$ is nilpotent. Thus $Z(G)\langle x,y\rangle$ is
an extension of a central subgroup by a nilpotent group, and so is nilpotent;
and so $\langle x,y\rangle$ is nilpotent, as required. \qed
\end{proof}

\begin{question}
What, if anything, can be said for infinite groups?
\end{question}

\subsection{Other graphs}

All the graphs studied so far have the property that two group elements 
which generate the same cyclic subgroup are closed twins. So it would be very
natural to collapse them by factoring out this equivalence relation.
Alternatively, one could simply remove edges between such pairs, so that they
become open twins. Note that the original, the quotient, and the graph with
edges removed all have the same cokernel; so, if one of them is a cograph,
then they all are.

We could put a graph at the bottom of the hierarchy, in
which $x\sim y$ if $\langle x\rangle=\langle y\rangle$; then the second
possibility suggested above fits into our scheme as the difference between
this graph and one of the others.

I end this section with a general question.

\begin{question}
For which types of graph, and which groups, is the relation $\equiv$ given
by $x\equiv y$ if $\langle x\rangle=\langle y\rangle$ definable directly
from the graph without reference to the group?

In particular, if the cokernel of the graph is equal to the quotient by
the equivalence relation $\equiv$, this will be true. For which groups is
this the case?
\end{question}

\section{Intersection graphs}
\label{s:intersection}

There turns out to be a close connection between certain intersection graphs
defined on $G$, and some of the graphs in our hierarchy. First I look briefly
at the connection in the abstract, then discuss some particular cases.

\subsection{Dual pairs}

Let $\mathrm{B}$ be a bipartite graph. If it is connected, it has a unique
bipartition: take a vertex $v$; then the bipartite blocks are the sets of
vertices at even (resp.~odd) distance from $v$. If $\mathrm{B}$ is not
connected, the bipartition is not unique; in fact, there are $2^{\kappa-1}$
bipartitions, where $\kappa$ is the number of connected components, since we
can make a bipartite block by choosing a bipartite block in each component
and taking their union. However, I will always assume that the bipartition of
$\mathrm{B}$ is given, and is part of its structure.

The \emph{halved graphs} arising from $\mathrm{B}$ are the graphs $\Gamma_1$
and $\Gamma_2$ whose vertex set is a bipartite block, two vertices adjacent
in the relevant graph if and only if they lie at distance $2$ in $\mathrm{B}$.

We call a pair of graphs $\Gamma_1$ and $\Gamma_2$ a \emph{dual pair} if
there is a bipartite graph $\mathrm{B}$ without isolated vertices such that
$\Gamma_1$ and $\Gamma_2$ are the halved graphs of $\mathrm{B}$.

I warn that this concept is not the same as the vague notion of duality which
informed the name ``dual enhanced power graph''. It is however closely
connected with duality in design theory and geometry, or between a graph and
the linear hypergraph (or partial linear space) that we used in
Section~\ref{s:indgg}.

\begin{prop}\label{p:dualpair}
Let $\Gamma_1$ and $\Gamma_2$ be a dual pair of graphs. Then $\Gamma_1$ is
connected if and only if $\Gamma_2$ is connected. More generally, there is a
natural bijection between connected components of $\Gamma_1$ and connected
components of $\Gamma_2$ with the property that corresponding components have
diameters which are either equal or differ by~$1$.
\end{prop}

\begin{proof}
Any vertex of $\Gamma_1$ is joined (by an edge of $\mathrm{B}$) to a vertex
of $\Gamma_2$, and \emph{vice versa}, since $\mathrm{B}$ has no isolated
vertices. Now suppose that two vertices of $\Gamma_1$ are joined by a path of
length $d$. Then there is a path of length $2d$ in $\mathrm{B}$ joining them.
So a connected component of $\mathrm{B}$ is the union of corresponding
connected components in $\Gamma_1$ and $\Gamma_2$. Suppose that a component
of $\Gamma_1$ has diameter $d$. Take two vertices $v_1$, $v_2$ in the
corresponding component of $\Gamma_2$. Choose vertices $u_1$ and $u_2$ of
$\Gamma_1$ joined in $\mathrm{B}$ to $v_1$ and $v_2$ respectively. These
two vertices lie at distance $r\le d$, say; so there is a path of length
at most $2r$ in $\mathrm{B}$ joining them. Thus $v_1$ and $v_2$ have distance
at most $2r+2$ in $\mathrm{B}$, whence their distance in $\Gamma_2$ is at most
$r+1$, hence at most $d+1$. So the diameter of a component of $\Gamma_2$ has
diameter at most one more than the corresponding component of $\Gamma_1$.
Interchanging the roles of the dual pair completes the proof. \qed
\end{proof}

\begin{question}
What other relations hold between properties of a dual pair of graphs?
\end{question}

If the bipartite graph is \emph{semiregular}, then a number of properties
transfer between the corresponding dual pair, especially spectral properties,
and (related to this) optimality properties of statistical designs~\cite{bc}.

\subsection{Graphs on groups and intersection graphs}

In order to apply this result, I give a general construction showing that
certain graphs defined on the non-identity elements of a group form dual
pairs with certain intersection graphs of families of subgroups.

\begin{prop}
Let $G$ be a finite non-cyclic group, and let $\mathcal{F}$ be a family of
non-trivial proper subgroups of $G$ with the property that its union is $G$.
Let $\Gamma$ be the graph defined on the non-identity elements of $G$ by
the rule that $x$ is joined to $y$ if and only if there is a subgroup
$H\in\mathcal{F}$ with $x,y\in H$. Then $\Gamma$ and the intersection graph
of $\mathcal{F}$ form a dual pair.
\end{prop}

\begin{proof}
We form the bipartite graph $\mathrm{B}$ whose vertex set is
$(G\setminus\{1\})\cup\mathcal{F}$, where a group element $x\ne1$ is joined
to a subgroup $H\in\mathcal{F}$ if and only if $x\in H$. We verify the
conditions for a dual pair.

First, $\mathrm{B}$ has no isolated vertices: for each subgroup in 
$\mathcal{F}$ is non-trivial, so contains an element of $G\setminus\{1\}$,
and every such element is contained in a subgroup in $\mathcal{F}$, since
the union of this family is $G$.

Next, two subgroups are joined in the intersection graph if and only if their
intersection is non-trivial (that is, contains an element of $G\setminus\{1\}$;
and, by assumption, two non-trivial elements are adjacent if and only if some
element of $\mathcal{F}$ contains both.

Note that we have assumed that $G$ is non-cyclic; this in fact follows from
the fact that it is a union of proper subgroups, since a generator would lie
in no proper subgroup. \qed
\end{proof}

\subsection{Applications}

I will consider several cases. I begin with the ``classical'' case, where
the vertices are all the non-trivial proper subgroups of $G$, joined if two
vertices are adjacent. These were first investigated by Cs\'ak\'any and
Poll\'ak, who considered non-simple groups; they determined the groups
for which the intersection graph is connected and showed that, in these
cases, its diameter is at most~$4$. For simple groups, Shen~\cite{shen}
showed that the graph is connected and asked for an upper bound; Herzog
\emph{et al.}~\cite{hlm} gave a bound of $64$, which was improved to $28$
by Ma~\cite{ma}, and to the best possible $5$ by Freedman~\cite{freedman},
who showed that the upper bound is attained only by the Baby Monster and
some unitary groups (it is not currently known exactly which).

\begin{prop}
Let $G$ be a non-cyclic finite group. Then the induced subgraph of the
non-generating graph of $G$ on non-identity elements and the intersection graph
of $G$ form a dual pair.
\end{prop}

\begin{proof}
Take $\mathcal{F}$ to be the family of all non-trivial proper subgroups of $G$.
\qed
\end{proof}

So the reduced non-generating graph of a non-abelian finite simple group has
diameter at most~$6$; this bound can be reduced to $5$, and possibly to $4$,
perhaps with specified exceptions (Saul Freedman, personal communication).

\medskip

Now we turn to the commuting graph.

\begin{prop}
Let $G$ be a finite group with $Z(G)=1$. Then the reduced commuting graph of
$G$ (on the vertex set $G\setminus\{1\}$) and the intersection graph of
non-trivial abelian subgroups of $G$ form a dual pair.
\end{prop}

\begin{proof}
The condition $Z(G)=1$ ensures that the reduced commuting graph does have
vertex set $G\setminus\{1\}$, and also implies that $G$ is not cyclic. Take
$\mathcal{F}$ to be the family of all non-trivial abelian subgroups of $G$
(all are proper subgroups since $G$ is not abelian). Two elements are joined
in the reduced commuting graph if and only if the group they generate is
abelian. \qed
\end{proof}

\begin{corollary}
For a finite group $G$ with $Z(G)=\{1\}$, the following four conditions are
equivalent:
\begin{enumerate}\itemsep0pt
\item the Gruenberg--Kegel graph of $G$ is connected;
\item the reduced commuting graph of $G$ is connected;
\item the intersection graph of non-trivial abelian subgroups of $G$ is
connected;
\item the intersection graph of maximal abelian subgroups of $G$ is
connected.
\end{enumerate}
\end{corollary}

\begin{proof}
The equivalence of (a) and (b) comes from Theorem~\ref{t:connectedcom}, and
that of (b) and (c) from Proposition~\ref{p:dualpair}. For the equivalence of
(c) and (d), note that any non-trivial abelian subgroup is contained in a 
maximal abelian subgroup, to which it is joined, so (d) implies (c). The
converse holds because any path in the intersection graph of non-trivial
abelian subgroups can be lifted to a path in the intersection graph of
maximal abelian subgroups. \qed
\end{proof}

\begin{prop}
Let $G$ be a group which is not cyclic or generalised quaternion. Then the
induced subgraph of the enhanced power graph of $G$ on the set of non-identity
elements and the intersection graph of non-trivial cyclic subgroups of $G$
form a dual pair.
\end{prop}

In fact the theorem applies also to generalised quaternion groups; but for
these, both the reduced enhanced power graph and the intersection graph are
connected for the trivial reason that they contain a vertex joined to all 
others.

\begin{proof}
We take $\mathcal{F}$ to be the family of non-trivial cyclic subgroups of
$G$. \qed
\end{proof}

Note that, if we take the graph $\DEP(G)$ and collapse the equivalence classes
of the relation $\equiv$, where $x\equiv y$ if
$\langle x\rangle=\langle y\rangle$, we obtain the intersection graph of
non-trivial cyclic subgroups of $G$. (This is probably why it was called the
``intersection graph'' in \cite{cs}.)

A study of intersection graphs of cyclic subgroups has been published by
Rajkumar and Devi~\cite{rd}.

\section{More general graphs}

When we think about graphs on groups, we want there to be some connection
between the graph and the group. This connection is mostly expressed in
terms of invariance of the graph under something, either right translations
or automorphisms of the group. The first gives rise to Cayley graphs, as
discussed briefly in Section~\ref{s:cayley}.

So the focus here is on graphs on a group $G$ invariant under the automorphism
group $\Aut(G)$ of $G$. We have seen that all graphs in the hierarchy do
satisfy this condition.

There are several ways we could approach the general case.
\begin{itemize}\itemsep0pt
\item Any graph invariant under $\Aut(G)$ is a union of orbital graphs for
$\Aut(G)$.
\item We could define the adjacency in the graph by a first-order formula
with two free variables.
\item We could define adjacency by some more recondite group-theoretic property.
\end{itemize}
We will see examples below.

However, it matters whether we are defining the graph on a single group, or
defining it on the class of all groups.

\subsection{On a specific group}

If we are given a group $G$, and can compute $\Aut(G)$, then the first
procedure (taking unions of orbital graphs) obviously gives all orbital
(di)graphs for $G$.

\begin{theorem}
Given a group $G$, for every $\Aut(G)$-invariant graph, there is a formula
$\phi$ in the first-order language of groups such that $x\sim y$ if and only
if $G\models\phi(x,y)$.
\end{theorem}

\begin{proof} By the so-called Ryll-Nardzewski Theorem, proved also by Engeler
and by Svenonius (see~\cite{hodges}), $G$ is oligomorphic, so the $G$-orbits
on $n$-tuples are $n$-types over $G$, that is, maximal sets of $n$-variable
formulae consistent with the theory of $G$; but all types are principal, so
each is given by a single formula. \qed
\end{proof}

\begin{question}
Given $G$, is there a bound for the complexity of the
formulae defining orbital graphs for $\Aut(G)$ acting on $G$ (for example,
for the alternation of quantifiers)?
\end{question}

Clearly the commuting graph can be defined by the quantifier-free formula
$xy=yx$. If $G$ is an elementary abelian $2$-group, there are only three
(non-diagonal) orbital graphs, defined by the formulae $(x=1)\wedge(y\ne1)$,
$(x\ne 1)\wedge(y=1)$, and $(x\ne1)\wedge(y\ne1)\wedge(x\ne y)$ respectively.

\subsection{For classes of groups}

As we have seen, the commuting graph is defined uniformly for all groups by
the quantifier-free formula $xy=yx$. 

It seems unlikely that the other graphs listed earlier have uniform
first-order definitions. The statement $\langle x,y\rangle=G$ seems to
require quantification either over words in $x,y$ or over subsets of $G$, and
so to need some version of higher-order logic for its definition.

For example, suppose that there is a formula $\phi(x,y)$ which, in any
finite group, specifies that $x$ and $y$ are joined in the power graph.
Taking $C_n$ with $n$ even, with $x$ a generator and $y$ of order $2$,
the formula is always satisfied. So it should hold in an ultraproduct of
such groups (see~\cite{bs}). But in the ultraproduct, $x$ has infinite order
and $y$ has order~$2$, so $y$ cannot be a power of $x$.

\subsection{Applications}

If we are given a specific group $G$ and know its automorphism group, then
constructing all the orbital graphs is a simple polynomial-time procedure.

If we are given $G$ and don't know (and maybe are trying to find out about)
its automorphism group, then clearly some indication of which first-order
formulae need to be considered would be helpful. Maybe, given $g,h\in G$,
the type of $(g,h)$ (the set of $\phi(x,y)$ such that $G\models\phi(g,h)$)
could be described, and a formula generating the type found.

Another situation that might arise would be that we are given one or more
graphs defined on general groups and are interested to know for which groups
they have some property, e.g. two graphs equal. If we had first-order
descriptions of the graphs, we would just be looking for models of some
first-order sentence.

\section{Infinite groups}

The definitions of the graphs in the hierarchy (with the exception of the deep
commuting graph), and the inclusions among them, work without change for
infinite groups. I will simply mention a few highlights here, as there is
little in the way of general theory.

\subsection{Power graph and directed power graph}

Let $p$ be a prime number. The \emph{Pr\"ufer group} $G=C_{p^\infty}$ is the
group of rational numbers with $p$-power denominators mod~$1$, or the
multiplicative group of $p$-power roots of unity. Every element of the group
has $p$-power order, and the group has a unique subgroup of order $p^n$ for
any $n$. It follows that the power graph of $G$ is a countable complete graph,
independent of the choice of prime.

The directed power graph does determine the prime, since the class of
elements immediately above the identity has size $p-1$. This shows that
the power graph does not determine the directed power graph for infinite
groups in general.

However, the implication does hold for torsion-free groups. This was shown
by Zahirovi\'c~\cite{zahir}; a preliminary result appears in~\cite{cgj}.
In fact, the hypotheses in Zahirovi\'c's result are weaker; I refer to the
paper for details, which show the important role played by the Pr\"ufer groups
in this problem.

\subsection{Independence number}

Perhaps the most striking result on the commuting graph of an infinite group is
the following, due to Bernhard Neumann (answering a question of Paul
Erd\H{o}s):

\begin{theorem}\label{t:neumann}
Let $G$ be an infinite group. Then the following are equivalent:
\begin{enumerate}\itemsep0pt
\item $\Com(G)$ has no infinite coclique;
\item there is a finite upper bound on the size of cocliques in $\Com(G)$;
\item $Z(G)$ has finite index in $G$.
\end{enumerate}
\end{theorem}

I have stated Neumann's result like this for comparison with what follows. He
proved that (a) implies (b) and (c). Now (b) implies (a) is trivial, and
(c) implies (b) because if (c) holds, then $G$ is a finite union of abelian
subgroups (since $\langle Z(G),g\rangle$ is abelian for all $g\in G$), and a
coclique in $\Com(G)$ can contain at most one vertex from each subgroup.

What about the power graph or enhanced power graph?

Certainly, if either of these graphs has no infinite coclique, then neither
does $\Com(G)$; so $Z(G)$ has finite index in $G$. But consider the group
$G=C_{p^\infty}\times C_{q^\infty}$, where $p$ and $q$ are distinct primes.
It is easy to show that $\Pow(G)$ has no infinite coclique; but, if
$a_n$ has order $p^n$ and $b_n$ has order $q^n$, then
\[\{a_nb_0,a_{n-1}b_1,\ldots,a_1b_{n-1},a_0b_n\}\]
is a coclique of size $n+1$, for any $n$. 

However, if $\Pow(G)$ has no infinite coclique, then $G$ is the union of
finitely many abelian subgroups. So we first ask, which abelian groups have
no infinite coclique? Such a group must be a torsion group; for, if $a$ were
an element of infinite order, then $\{a^p:p\hbox{ prime}\}$ is an infinite
coclique in the power graph. There can only be finitely many primes such that
$G$ contains element of order~$p$. So $G$ is the direct sum of its finitely many
Sylow subgroups. Moreover, the Sylow subgroups must have finite rank. So we
can conclude:

\begin{theorem}
Let $G$ be an infinite group. Then the following are equivalent:
\begin{enumerate}\itemsep0pt
\item $\Pow(G)$ has no infinite coclique;
\item $Z(G)$ has finite index in $G$ and is a direct sum of finitely many
$p$-torsion subgroups of finite rank, for primes $p$.
\end{enumerate}
\end{theorem}

So $G$ is locally finite, a result of Shitov~\cite{shitov}.

If we make the stronger hypothesis that the size of cocliques is
bounded, then we can strengthen the conclusion to assert that all but one
of the Sylow subgroups of $Z(G)$ is finite.

For the enhanced power graph, Abdollahi and Hassanabadi~\cite{ah1} proved that
the analogue of Neumann's Theorem does hold:

\begin{theorem}
Let $G$ be an infinite group. Then the following are equivalent:
\begin{enumerate}\itemsep0pt
\item $\EPow(G)$ has no infinite coclique;
\item there is a finite upper bound for the size of cocliques in
$\EPow(G)$;
\item $Z_{\EPow}(G)$ has finite index in $G$.
\end{enumerate}
\end{theorem}

(Recall that $Z_{\EPow}(G)$ is the \emph{cyclicizer} of $G$.)

\subsection{Cliques and colourings of the power graph}

Another significant body of work on power graphs concerns the clique parameters.
Here are some striking results, which appear in \cite{aetal,cj,shitov}.

\begin{theorem}
Any infinite group has clique number and chromatic number at most countable.
\end{theorem}

\begin{theorem}
For an infinite group $G$, the following conditions are equivalent:
\begin{enumerate}\itemsep0pt
\item $\Pow(G)$ has finite clique number;
\item $\Pow(G)$ has finite chromatic number;
\item $\EPow(G)$ has finite clique number;
\item $\EPow(G)$ has finite chromatic number;
\item $G$ is a torsion group with finite exponent.
\end{enumerate}
\end{theorem}

\begin{proof}
The power graph of an infinite cyclic group $\langle g\rangle$ contains an
infinite clique $\{g^{2^n}:n\ge0\}$. So a group satisfying any of the first
four conditions is a torsion group. Now the results are proved just as for
finite groups in Section~\ref{s:cc}. \qed
\end{proof}

The cited papers contain other miscellaneous results.

\subsection{Cographs}
Cographs work quite differently in the infinite case. The two definitions
earlier (a graph containing no induced $P_4$, and a graph built from the
$1$-vertex graph by disjoint union and complementation) are no longer
equivalent, even if infinite disjoint unions are allowed.
Covington~\cite{covington} constructed a countable $P_4$-free graph which
is isomorphic to its complement (so it and its complement are both connected).
This remarkable object also has a high degree of symmetry.

\section{Beyond groups}

The ideas behind some of these graphs can be extended to other algebraic
structures.

A \emph{magma} is a set with a binary operation. (The term \emph{groupoid} is
sometimes used, but I will avoid this since it is also used for a category in
which every morphism is invertible.) Beyond groups, the two classes of magmas
most studied are \emph{semigroups} (satisfying the associative law) and
\emph{quasigroups} (in which left and right division are well-defined), and
in particular \emph{monoids} and \emph{loops} (semigroups, resp.\ quasigroups,
with identity elements).

Clearly the definition of commuting graph makes sense in any magma. For the
power graph and its relatives, it is necessary to make sense of powers of an
element. One can define \emph{left powers} inductively by $a^1=a$ and
$a^{n+1}=a\circ a^n$ for $n\ge1$. (This is the approach adopted
in~\cite{walker}.) An alternative is to restrict to \emph{power-associative
magmas}, those in which the product of $n$ terms each equal to an element $a$
is independent of the bracketing used to evaluate it. Now, just as for groups,
we have:

\begin{prop}
In a power-associative magma, the directed power graph is a partial preorder,
and so the power graph is the comparability graph of a partial order.
\end{prop}

\begin{question}
For which magmas, or quasigroups, is the power graph (defined using left
powers) a comparability graph of a partial order, or a perfect graph?
\end{question}

The power graph of a semigroup was defined early in their study of power
graphs, see \cite{kq2}).

The commuting graphs of semigroups are considered by Ara\'ujo
\emph{et al.}~\cite{akk}, who pose a number of questions about them.

As for groups, the power graph of a semigroup is a spanning subgraph of
its commuting graph. The enhanced power graph could be defined for any
semigroup; to my knowledge this has not been studied. It is not clear whether
the definition of the deep commuting graph could be adapted for semigroups.

Commuting graphs of semigroups are universal (and power graphs are universal
for comparability graphs of partial orders), since these statements hold for
groups.

The intersection graph of the subsemigroups of a semigroup had been studied
much earlier: Bos\'ak~\cite{bosak} raised the question of its connectedness
in 1963, and the question was soon resolved by Lin~\cite{lin} and
Pond\v{e}li\v{c}ek~\cite{pond}: for any finite semigroup, this graph is
connected with diameter at most~$3$. These results preceded the investigation
of the intersection graph for groups, mentioned earlier. (The intersection
graphs of semigroups are not comparable with the
intersection graphs of groups, since any subgroup of a group contains the
identity, so adjacency requires their intersection to be non-trivial.) 

The commuting graph (and its complement), the intersection graph of cyclic
subgroups, and the power graph have also been studied for quasigroups and
loops, especially for the classes of Moufang and Bol loops: see, for
example, \cite{ahmad,ha,hi,walker}. Moufang loops form a class of loops
which is perhaps closest to groups: a Moufang loop is a loop satisfying
the identity $z(x(zy)) = ((zx)z)y$ (a weakening of the associative law).
In particular, a $2$-generated subloop of a Moufang loop is associative; so
a Moufang loop is power-associative, and its power graph is the comparability
graph of a partial order.

A question raised some time ago but to my knowledge not yet answered is:

\begin{question}
If the power graphs of two finite Moufang loops are isomorphic, are their
directed power graphs isomorphic?
\end{question}

This question has been investigated by Nick Britten, who has found no 
counterexamples among the Moufang loops in the LOOPS package~\cite{loops}
for \textsf{GAP} (Michael Kinyon, personal communication).

Going beyond a single binary operation, we reach the class of rings.
For these, the \emph{zero-divisor graph} of a ring was introduced by
Beck~\cite{beck} in 1988: the vertices are the ring elements, with $a$ and $b$
joined whenever $ab=0$. If the ring is commutative, the graph is undirected.
Another graph associated with a ring is the \emph{unit graph}, in which $a$
and $b$ are joined whenever $a+b$ is a unit~\cite{ampy}.

\end{document}